\documentclass[twoside]{amsart}
\usepackage[utf8]{inputenc}
\usepackage[english]{babel}
\usepackage{amssymb,amsmath,amsthm}
\usepackage{graphicx}
\usepackage[small,nohug,heads=vee]{diagrams}
\diagramstyle[labelstyle=\scriptstyle]
\usepackage{xspace}
\usepackage{enumerate}
\usepackage{geometry} 
\newcommand{\F}[1]{\ensuremath{\mathbb{F}_{#1}}}

\newcommand{\QQ}{\ensuremath{\mathbb{Q}}\xspace}

\newcommand{\ZZ}{\ensuremath{\mathbb{Z}}\xspace}
\newcommand{\Zp}{\ensuremath{\mathbb{Z}_{(p)}}\xspace}
\newcommand{\NN}{\ensuremath{\mathbb{N}}\xspace}

\newcommand{\rarr}{\rightarrow}

\newcommand{\xrarr}[1]{\xrightarrow{#1}}

\newcommand{\id}{\ensuremath{\mbox{id}}}

\newcommand{\ot}{\otimes}

\theoremstyle{definition}
\newtheorem{Def}{Definition}[section]
\newtheorem{Rk}[Def]{Remark}
\newtheorem*{Exl}{Example}
\theoremstyle{plain}
\newtheorem{Th}[Def]{Theorem}
\newtheorem{Prop}[Def]{Proposition}
\newtheorem{Cr}[Def]{Corollary}
\newtheorem{Lm}[Def]{Lemma}
\newtheorem*{BigTh}{Classification of Operations Theorem  (COT)}
\newtheorem*{BigTh-add}{Algebraic Classification of Additive Operations Theorem  (CAOT)}

\begin{document}

\title{Chern classes from Morava K-theories to Chow groups}
\author{Pavel Sechin}
\address{National Research University Higher School of Economics,  Russian Federation}
\email{pasechnik771@gmail.com}
\date{}

\begin{abstract}
In this paper we calculate the ring of unstable (possibly non-additive) operations
from algebraic Morava K-theory $K(n)^*$ to Chow groups with $\Zp$-coefficients.
More precisely, we prove that it is a formal power series ring on generators $c_i:K(n)^*\rarr CH^i\ot\Zp$,
which satisfy a Cartan-type formula.
\end{abstract}

\maketitle

\section*{Introduction}

(Complex) orientable cohomology theories appeared in algebraic topology as rigid examples
of generalized cohomology theories. The very notion of orientability is 
strongly connected to the notion of Chern classes 
of complex vector bundles by the following.
Loosely speaking, orientable theories are those theories for which all invariants
of vector bundles (aka operations from K-theory) 
are expressed in terms of their Chern classes
and all universal relations between these invariants come from relations between vector bundles.

In algebraic geometry the notion of an orientable cohomology theory did not seem to appear
until the fundamental work of F.Morel and V.Voevodsky on the motivic homotopy theory.
The setting developed by them allowed geometers 
to 'borrow' many notions and constructions from topology
in a clear conceptual (but technically uneasy) way. In particular, V.Voevodsky
has performed a motivic construction of the Thom spectrum to introduce
the theory of algebraic cobordisms $MGL^{*,*}$ which was later proved to be
the universal orientable theory.

The notion of orientability of more general cohomology theories than those
which are
representable in the stable motivic homotopy category
was studied by I.Panin and A.Smirnov. They observed that 
orientability has three equivalent avatars: it can be specified either by Thom classes of line bundles,
Chern classes of line bundles or pushforward maps for proper morphisms.
Perhaps, it shows crucial importance of Chern classes 
in defining orientable theories.

While the study of oriented motivic spectra is not easy,
any motivic-representable cohomology theory 
has a {\it pure} part, sometimes referred to as a {\it small} theory 
as opposite to the whole {\it big} theory. 
Small theories are presheaves of graded  rings which
often allow for a more geometric description in comparison with big theories. 
For example, the pure part of motivic cohomology is Chow groups, and
the pure part of algebraic K-theory is the Grothendieck group of vector bundles 
$K_0[\beta,\beta^{-1}]$.

The pure part of algebraic cobordisms $MGL^{*,*}$ has been developed 
in a seminal paper by M.Levine and F.Morel. 
It is usually denoted as $\Omega^*$ and is reffered 
to as Levine-Morel algebraic cobordism. 
(However, the comparison of $\Omega^*$ and $MGL^{*,*}$ turned out
to be difficult and was proved only years later).
They also gave a definition of an orientable theory 
which is different than that of Panin-Smirnov,
for which they proved that $\Omega^*$ is the universal oriented theory.
This universality allowed to introduce the whole bunch of orientable
theories which were investigated before in algebraic topology.
More precisely, for any formal group law $F_R$ over any ring $R$
there exist an oriented theory $\Omega^*\ot_{\Omega^*(k)} R$
with the ring of coefficients $R$ and the corresponding formal group law $F_R$.
These theories are called {\it free theories}.

In particular, one is able to introduce small Morava K-theories $K(n)^*$
as free theories with the ring of coefficients $\Zp$ and Lubin-Tate formal group law
(prime $p$ is not usually included in the notation, for each $p$ and $n\ge 1$ there is a theory $K(n)^*$).
Their topological counterpart (perhaps, with $\F{p}$-coefficients)
 appeared in the chromatic homotopy theory,
a study of the stable homotopy category localized away from prime $p$.
In algebraic geometry big Morava K-theories conjectural at that moment
 were used in the course of the proof of the Bloch-Kato conjecture 
by V.Voevodsky (eventually in a disguised way). 
Recently Morava K-theories were used 
to study projective homogenous varieties (\cite{PS_PetSem}). 

It would not be a big lie to say that Morava K-theories 
are quite mysterious objects. There is though at least one reason
why they are called {\sl K-theories}, which is that $K(1)^*$ 
can be viewed as theory $K_0\ot\Zp$ with its orientation changed.
The goal of this paper is to show that another similarity exists.
Note that usual Chern classes can be considered as (non-additive) operations
from $K_0$ to an oriented theory $A^*$.
If $A^*$ is a presheaf of $\Zp$-algebras, one can extend 
Chern classes to operations from $K_0\ot\Zp$, or equivalently $K(1)^*$, to $A^*$.
In this paper we produce similar operations from $K(n)^*$ for all $n$
when $A^*=CH^*\ot\Zp$. 

Chern classes
have two important properties as operations: first, there are no relations between polynomials on them,
and second, any operation from $K_0$ to any orientable theory $A^*$ can be written in terms of them.
 Probably, one can characterize 
orientable theories by this features having added some axioms in the definition of a theory.
In view of this the following question naturally arises:
does there exist a notion of $K(n)^*$-orientable theories?
Any such theory $A^*$ should be equipped with some operations $K(n)^*\rightarrow A^*$ 
which imitate usual Chern classes and can be named likewise.
These Chern classes should freely generate the ring 
of all operations to $A^*$ and we might try 
to reformulate some properties of $A^*$ in terms of them.

Although, we are not yet in the position to answer this general question, 
we provide some evidence that the question is not completely senseless.
Our main result is the existence of a series of operations
 $c_i:K(n)^*\rightarrow {CH^i\otimes\mathbb{Z}_{(p)}}$
which satisfy the Cartan formula 
and generate freely all operations to ${CH^*\otimes\mathbb{Z}_{(p)}}$.
We may interprete this as $CH^*\otimes\mathbb{Z}_{(p)}$ being '$K(n)^*$-orientable'.
At the same time we are able to show that 
there are many more operations from $K(n)^*$ to $CH^*/p$, except from those generated by operations $c_i$.
It should be noted that operations 
$c_1, c_2, \ldots, c_{p^n}$ appeared previously in a paper by V.Petrov and N.Semenov 
(cf. Remark \ref{rem:petsem}).

We need to say a few more words about the technique used in the proof of our result.
Calculating operations seems to be a complicated task in algebraic geometry.
However, recently the work of A.Vishik made this problem amenable 
for operations between small orientable theories. 
Vishik has developed a definition of orientable theories of rational type, for which it is possible
to construct the value of a theory on a particular variety by induction on its dimension.
It turned out that these theories are precisely free theories appearing above.
The inductive description allows to classify completely all operations
 from a free theory to any orientable theory.
More precisely, operation can be uniquely constructed by its restriction 
to the values on products of projective spaces s.t. it commutes 
with pullbacks along several morphisms between them.
These morphisms are the Segre embeddings, the diagonal maps,
the point inclusions, the projections and the permutations.
Using the projective bundle theorem (over a point) one can make a 1-to-1 correspondence between operations
and solutions of a certain system of linear equations which depends {\sl solely} on 
formal group laws of theories in consideration.

Our construction of operations from $K(n)^*$ to $CH^*\ot\Zp$ 
is just a solution of a system provided by Vishik's theorem and
therefore it is rather technical and lacks any geometric flavor at all.
We wish these operations could be constructed in a more elegant and meaningful way.
However, we think the calculation is a not-uninteresting example which shows that 
system of equations from Vishik's theorem is computable
and that the notion of $K(n)^*$-orientability could exist.

{\bf Outline.} In Section \ref{prelim} 
we recall notions of generalized orientable cohomology theory,
free theories of Levine-Morel and state Vishik's Classification of Operations Theorem.
In Section \ref{sec_op_K0} we classify all operations and all additive operations from K-theory 
to orientable theories. 
This is both used as a motivation for calculating operations from higher Morava K-theories 
and as a tool to simplify several proofs.

In Section \ref{sec_adams} we prove that Adams operations (existing on any free theory as proved by Vishik)
are central with respect to all operations. Using it we are able to extend non-additive 
operations when localizing coefficients (Proposition \ref{prolong}), namely,
we show that 
any operation $K_0\rarr CH^*\ot \Zp$ factors uniquely through the map $K_0\rarr K_0\ot\Zp$.
As a corollary of this statement we are able to compute the dimension of spaces of 
operations and additive operations from theories alike $K(n)^*$ to $CH^*\ot\Zp$.

Section \ref{sec_op_mor} contains main results of the paper.
Definitions and general properties of Morava K-theories are discussed in Section \ref{subsec_mor}.
In Section \ref{subsec_mainth} we state the main result, Theorem \ref{main}, and outline the general idea of its proof. 
Section \ref{sec_cor} contains several corollaries of the theorem which give a clue of how 
to prove it. Inductive construction of Chern classes is performed in section \ref{chern_constr}.
The construction itself is conditional upon a certain statement about additive operations
from $K(n)^*$ to $CH^*/p$. 
This statement is proved in Section \ref{op_modp} where we study additive operations from $K(n)^*$ to $CH^*/p$.
We find that there are many of them which can not be lifted to $CH^*\ot \Zp$ and prove a sufficient property
 for liftability as an additive operation. Section \ref{chern-base} closes the paper with the proof that constructed Chern classes
generate freely the ring of all operations to $CH^*\ot\Zp$.

\section*{Acknowledgements}

This work has started while the author visited A.~Vishik at the University of Nottingham in the Spring of 2015.
It is both encouragement given by him and persistent discussions with him 
which made this work appear. 
The author is enormously grateful to A.~Vishik for being his advisor and mentor during that time
as well as for useful comments on drafts of this text.

The study has been funded by the Russian Academic Excellence Project '5-100' 
and supported in part by Simons Foundation.

\section{Preliminaries}\label{prelim}

In this section we recall the notions of a generalized orientable cohomology theory,
a theory of rational type and state the main tool needed for our paper,  
Vishik's Classification of Operations Theorem.

Fix a field $k$ with $\mbox{char }k=0$.

\subsection{Orientable theories}

Several authors have inroduced slightly different notions of orientation of cohomology theories
 in algebraic geometry. 
 These include definitions of I.~Panin and A.~Smirnov(\cite[2.0.1]{PS_PanSmi})
 and of M.~Levine and F.~Morel \cite[1.1.2]{PS_LevMor}.
We will use a variant of the latter definition
with one more axiom (LOC) added as was used by A.~Vishik 
(note, however, that
this axiom is also denoted as (EXCI) in \cite{PS_Vish1}).

\begin{Def}[{Vishik, \cite[2.1]{PS_Vish1}}]\label{goct}
 
A generalized oriented cohomology theory (g.o.c.t.) is
a presheaf $A^*$ of commutative rings on a category of smooth varieties over $k$
supplied with the data of push-forward maps for proper morphisms.
Namely, for each proper morphism of smooth varieties $f:X\rarr Y$,
morphisms of abelian groups $f_*:A^*(X)\rarr A^*(Y)$ are defined.

The structure of push-forwards has to satisfy the following axioms 
(for precise statements see {\sl ibid}):
functoriality for compositions (A1), base change for transversal morphisms (A2),
projection formula, projective bundle theorem (PB),
 $\mathbb{A}^1$-homotopy invariance (EH)
and localisation axiom (LOC).
\end{Def}

Let us describe explicitly the axiom (LOC).
Let $X$ be a smooth variety over $k$, let $j:U\rarr X$ be an open embedding
and let $i:Z\rarr X$ be its closed complement. Define $A^*(Z)$ as a direct limit of $A^*(V)$
over the system of projective morhisms $f:V\rarr Z$ from smooth varieties $V$.
Push-forward maps $(i\circ f)_*$ induce the map $i_*:A^*(Z)\rarr A^*(V)$.

The localisation axiom says that the following sequence of abelian groups is exact:
\begin{equation}\tag{LOC}\label{eq_loc}
A^*(Z) \xrarr{i_*} A^*(X) \xrarr{j^*} A^*(U) \rarr 0
\end{equation}

{\it Notation}. Star in the superscript of a g.o.c.t. does not mean neccessarily
that the theory is graded. However, if it is, then we will freely replace the star 
by a number or a variable, e.g. $CH^2$ or $CH^i$.
In non-graded cases we do not drop the superscript in order to distinguish 
the theory itself with its ring of coefficients, i.e. $A^*$ is a presheaf of rings, and
$A$ is usually the corresponding value on a point, $A=A^*(Spec(k))$.
Nevertheless, we always write $K_0$ to denote the Grothendieck group of vector bundles 
as a g.o.c.t. with the ring of coefficients $\ZZ$.\\

For any g.o.c.t. one can define Chern classes of vector bundles in a usual way (\cite{PS_PanSmi})
as follows.
Let $X$ be a smooth variety, $L_X$ be a line bundle over $X$, 
denote by $s:X\rarr L_X$ the zero section of $L_X$. 
Then $c_1^A(L_X):=s^*s_*(1_X)$ is the first Chern class.
Higher Chern classes are defined using the projective bundle theorem
and are uniquely determined by the Cartan's formula 
and the property that $c_i^A(V)=0$ for $i>rk(V)$.

One may associate with each g.o.c.t.\ a formal group law (FGL) over its ring of coefficients.
Consider the infinite-dimensional projective space $\mathbb{P}^\infty$
as an ind-obejct of the category of smooth varieties. 
It is natural to consider the values of g.o.c.t. on products of 
$\mathbb{P}^\infty$ as inverse limits over the values on finite-dimensional
projective spaces.
From the (PB) axiom it follows then that 
 $A^*((\mathbb{P}^\infty)^{\times n})=A[[z_1,\ldots, z_n]]$
where $z_i$ denotes $c_1^A(\mathcal{O}(1)_i)$, 
and $\mathcal{O}(1)_i$ is the pullbacl of the canonical line bundle 
along the projection on the $i$-th component of the product.

The system of Segre maps between projective spaces
is compatible and yields a morphism of ind-objects 
$Seg:(\mathbb{P}^\infty)^{\times 2}\rarr \mathbb{P}^\infty$.
The pullback $Seg^*$ in $A^*$ is defined by its value on $z$
which we denote by $F_A(z_1,z_2)\in A[[z_1,z_2]]$. 
It is not difficult to show that $F_A$ is a formal group law.

Thinking of $\mathbb{P}^\infty$ as classifying line bundles
in a certain category (e.g. in the motivic stable homotopy category)
one notices that the Segre map induces tensor product of line bundles.
It is not surprising then that 
 for any line bundles $L, L'$
on a smooth variety $X$ the following formula holds: $c_1^A(L\otimes L')=F_A(c_1^A(L),c_1^A(L'))$ 
(\cite[Lemma 1.1.3]{PS_LevMor}).

\begin{Th}[Levine-Morel, {\cite[1.2.6]{PS_LevMor}}]
There exists the univeral g.o.c.t. $\Omega^*$ called algebraic cobordisms,
i.e. for any g.o.c.t. $A^*$ 
there exists a unique morphism of presheaves of rings
$p_A:\Omega^*\rarr A^*$ which respects the structure of pushforwards.

The associated formal group law of algebraic cobordisms
 is the universal formal group law,
i.e. $\Omega^*(\mathrm{Spec\ }k)$ 
is canonically isomorphic to the Lazard ring $\mathbb{L}$.
\end{Th}

\begin{Exl}
The Grothendieck group of vector bundles $K_0$ and Chow groups $CH^*$
are g.o.c.t.'s.

The corresponding formal groups are 
the multiplicative FGL $F_m(x,y)=x+y+xy$ and the additive FGL $F_a(x,y)=x+y$ over 
the ring $\mathbb{Z}$, respectively.
\end{Exl}

The unique morphism of oriented theories $\Omega^* \rarr A^*$ respects Chern classes,
and we know that the classes $c_i^\Omega(V)$ of a vector bundle $V$ on a variety $X$
are zero whenever $i>\dim X$. Hence,
the same is true for any g.o.c.t. The same argument shows that 
Chern classes are nilpotent in any g.o.c.t. for each particular variety.

One can construct a g.o.c.t. with any given FGL in the following way.

\begin{Def}[Levine-Morel,{ \cite[Rem. 2.4.14]{PS_LevMor}}]
Let $R$ be a ring, let $\mathbb{L}\rightarrow R$ be a ring morphism
 corresponding to a formal group law $F_R$ over $R$.

Then $\Omega^*\otimes_\mathbb{L} R$ is a g.o.c.t. which is called a {\bf free theory}.
Its ring of coefficients is $R$, and its associated FGL is $F_R$.
\end{Def}

\begin{Rk}
Note that {\it any} formal group law yields a g.o.c.t, which
is mainly due to the kind of localisation axiom in the definition.
The tensor product is exact on the right and it suffices for the property (LOC)
to stay true after the change of coefficients.

This shows the difference with orientable theories in topology
or in the stable motivic homotopy category
where additional conditions on the formal group law are imposed in order for it 
to be realized.
\end{Rk}

\begin{Th}[Levine-Morel, {\cite[Th. 1.2.2 and 1.2.3]{PS_LevMor}}]
Chow groups $CH^*$ and K-theory of vector bundles $K_0$ 
are free theories, i.e. natural morphisms
$$ \Omega^*\ot_{\mathbb{L}, F_a} \ZZ \rarr CH^*, \quad \Omega^*\ot_{\mathbb{L}, F_m} \ZZ \rarr K_0 $$
are isomorphisms.
\end{Th}

Theories of rational type were introduced by A.Vishik in \cite{PS_Vish1}
and are those g.o.c.t. which satisfy an additional axiom (CONST)
and a really strong 'inductive' property (but rather technical to state it here precisely).
One crucial feature of this property is that 
values on varieties can be described by induction on dimension,
and this allows to construct operations from such theories 
to orientable theories effectively, by induction on dimension.

The {\bf axiom (CONST)} for a g.o.c.t. $A^*$
says that for any smooth irreducible variety $X$
the natural restriction map $A\rarr A^*(k(X))$ is an isomorphism.
Here $A=A(k)$ and $A^*(k(X)):=\mathrm{colim_{U\subset X} }A^*(U)$ where
$U$ runs over all open subsets of $X$.

If this property is satisfied, the restriction maps $A\rarr A^*(X)\rarr A^*(k(X))$ permit
 to split $A^*$ as presheaf of abelian groups into two summands:
$A^*=\tilde{A}^*\oplus A$, where $A$ is a constant presheaf and $\tilde{A}^*$
 is an ideal subpresheaf of elements which are trivial in generic points.
The algebraic cobordism $\Omega^*$ satisfy this property,
and therefore all free theories as well.

The following theorem allows us to skip the definition of a theory of rational type,
which we will not use explicitly.

\begin{Th}[Vishik, {\cite[Prop. 4.9]{PS_Vish1}}]
Theories of rational type are precisely free theories.
\end{Th}

\subsection{Operations between theories of rational type}

Though the notion of a theory of rational type does not yield new examples of cohomology theories,
 their intrinsic inductive description
allows to study operations and poly-operations between them in a very efficient way.

\begin{Def}\label{operation}
Let $A^*, B^*$ be presheaves of abelian groups (or rings, or graded rings)
 on the category of smooth varieties over a field.

An {\bf operation} $\phi: A^*\rarr B^*$ is a morphism of presheaves of sets. 
The set of operations is denoted by $[A^*, B^*]$. 
If $B^*$ is a presheaf of rings, the set $[A^*, B^*]$ has the natural ring structure
induced by such structure on the target theory.

An {\bf additive operation} $\phi:A^*\rarr B^*$ is a morphism of presheaves of abelian groups.
The set of additive operations is denoted by $[A^*,B^*]^{add}$.
\end{Def}

We will need to consider not only operations, but poly-operations as well.
There are two types of them: external and internal ones.
 It is not hard to see that there is a 1-to-1 correspondence between these two notions (\cite[p.8]{PS_Vish2}).
As we will be concerned only with external operations, we omit the adjective in the following definition.

\begin{Def}
Let $A^*, B^*$ be presheaves of abelian groups (or rings, or graded rings)
 on the category of smooth varieties over a field $k$.
 
An {\bf $\mathbf{r}$-ary poly-operation} from $A^*$ to $B^*$ 
is a moprhism of presheaves of sets on the $r$-product category
of smooth varieties over a field $k$
from $(A^*)^{\times r}$ to $B^*\circ \prod^r$.

Explicitly, for varieties $X_1, \ldots, X_r$
poly-operation yields a map of sets 
$$ A^*(X_1)\times A^*(X_2) \times \ldots \times A^*(X_r) 
\rightarrow B^*(X_1 \times X_2 \times \ldots \times X_r).$$

The set of $r$-ary poly-operations is denoted by $[(A^*)^{\times r}, B^*\circ \prod^r]$.
\end{Def}

\begin{BigTh}[{Vishik, \cite[Th. 5.1]{PS_Vish1}, \cite[Th. 5.2]{PS_Vish2}}]\label{Vish_op}
\item 

Let $A^*$ be a theory of rational type and let $B^*$ be a g.o.c.t.

Then the set of $r$-ary poly-operations from $A^*$ to $B^*$
is in 1-to-1 correspondence with the following data:

maps of sets 
$\times_{i=1}^r A^*((\mathbb{P}^\infty)^{\times l_i})\rightarrow B^*(\times^r_{i=1}(\mathbb{P}^\infty)^{\times l_i})$ for $l_i \ge 0$
(restrictions of a poly-operation),
which commute with the pull-backs for:
\begin{enumerate}
\item the permutation action of a product of symmetric groups $\times_{i=1}^r \Sigma_{l_i}$;
\item the partial diagonals for each $i$;
\item the partial Segre embeddings for each $i$;
\item the partial point embeddings for each $i$;
\item the partial projections for each $i$.
\end{enumerate}
\end{BigTh}

\begin{Rk}\label{rem:grad}
If the target theory is graded, 
then the theorem allows one to compute poly-operations to each of the components of the target.

To see this note that grading on $B^*$ yields (additive) projectors $p_n:B^*\rarr B^n$
and an operation to a component $B^n$ is just an operation
which is zero when composed with $p_m$, $m\neq n$.
 As follows from the theorem, this property may be checked on products of projective spaces.
\end{Rk}

\begin{Rk}
Analogous result was proved in topology by T.Kashiwabara, \cite[Th. 4.2]{PS_Kash}.

We would expect that cohomology theories in topology
analogous to $K(n)^*$ and $CH^*\ot\Zp$ considered in this paper
satisfy conditions of Kashiwabara's theorem, however, we have not checked it. 
This would mean that our computation, Theorem \ref{main}, is true in topological context as well.
\end{Rk}

\subsection{Continuity of operations}\label{sec:cont}

Recall that by the (PB) axiom for any g.o.c.t. $A^*$
the value on the product of infinite-dimensional projective spaces
 $A^*((\mathbb{P}^\infty)^{\times l})$
 is a formal power series ring
$A[[c_1^A(\mathcal{O}(1)_1),\ldots, c_1^A(\mathcal{O}(1)_l)]]$,
where $\mathcal{O}(1)_i$ is the pull-back of the canonical line bundle from the $i$-th factor. 
Throughout the article we will use the notation $z_i^A := c_1^A(\mathcal{O}(1)_i)$,
or just $z_i$ if theory $A^*$ is clear from the context.

The restriction of any operation to products of projective spaces
satisfies the property of continuity which we now explain.

Let $G:A^*\rarr B^*$ be any operation from a theory of rational type  $A^*$ to any g.o.c.t. $B^*$.

By the COT $G$ is determined by maps of sets  
$G_{\{l\}}:A[[z^A_1,\ldots, z^A_l]]\rarr B[[z^B_1,\ldots, z^B_l]]$
for all $l\ge 0$.
As $G_{\{l\}}$'s have to commute with pull-backs along partial projections  
the following diagram is commutative for any $l\ge 0$:
\begin{diagram}
A[[z^A_1,\ldots, z^A_l]] &\rTo^{G_{\{l\}}} & B[[z^B_1,\ldots, z^B_l]] \\
\dInto & & \dInto \\
A[[z^A_1,\ldots, z^A_{l+1}]] &\rTo^{G_{\{l+1\}}} & B[[z^B_1,\ldots, z^B_{l+1}]]
\end{diagram}

This allows to use only one transform, the inductive limit of maps $G_{\{l\}}$
$$ G:A[[z^A_1,\ldots, z^A_l,\ldots]]\rarr B[[z^B_1, \ldots, z^B_l, \ldots]]$$
which uniquely determines $G_{\{l\}}$ for any $l$.

Denote by $F^k_A$ an ideal in $A[[z^A_1,\ldots, z^A_l, \ldots]]$ of series of degree $\ge k$
($F^k_B$ is defined analogously). Without loss of generality (as we may subtract $G(0)$),
 we may assume that $G(0)=0$.

\begin{Prop}[Vishik, {\cite[Prop. 5.3]{PS_Vish2}}]\ \\

Let $P, P'\in A[[z^A_1,\ldots, z^A_l, \ldots]]$ be s.t. $P \equiv P' \mod F^k_A$.
Then $$ G(P) \equiv G(P') \mod F^k_B. $$
\end{Prop}

This Proposition allows to calculate approximation of $G(P)$ approximating $P$.
In particular, an operation $G$ is determined by its restriction 
to the products of finite-dimensional projective spaces, or equivalently
by the maps
$$G_{r,n}:A^*((\mathbb{P}^n)^{\times l})=A[[z^A_1,\ldots, z^A_l]]/F_A^{n+1} 
\rarr B^*((\mathbb{P}^n)^{\times l})=B[[z_1,\ldots, z_l]]/F_B^{n+1}.$$
In other words, the maps $G_{\{l\}}$ are determined by their restriction 
to the polynomial rings $A[z^A_1,\ldots, z^A_l] \subset A[[z^A_1,\ldots, z^A_l]]$
and the same is true for the map $G$.

An analogous statement is true for poly-operations as well.

\subsection{Classification of additive operations}\label{sec:CAOT}

The data of an additive operation specified in the COT 
can be rewritten as a set of formal power series satisfying a system of linear equations.
We specify this system here.

Let $A^*$ be a theory of rational type, let $B^*$ be a g.o.c.t. and let $\phi$ be an additive operation 
from $A^*$ to $B^*$.
For $l\ge 0$ define maps ${G_l\in Hom_{Ab}(A, B[[z^B_1,\cdots, z^B_l]])}$ 
to be the values of the operation on 'basis' monomials of $z$-degree $l$:
$$G_l(\alpha)(z^B_1,\ldots, z^B_l) := \phi(\alpha z_1^A\cdots z_l^A).$$

It is clear by continuity that the maps of abelian groups
$G_{\{l\}}:A^*((\mathbb{P}^\infty)^{\times l}) \rarr B^*( ((\mathbb{P}^\infty)^{\times l}) )$
carry the same information as the maps $G_l$. 

\begin{BigTh-add}[Vishik, {\cite[Th. 6.2]{PS_Vish1}}]\ \\
Let $A^*$ be a theory of rational type and let $B^*$ be any g.o.c.t.
Denote by $F_A(x,y)=\sum_{i,j} a_{ij}x^iy^j$, $F_B=\sum_{i,j} b_{ij}x^iy^j$
the corresponding formal group laws.

Then the abelian group of additive operations $[A^*,B^*]^{add}$
is in 1-to-1 correspondence with the set of maps $G_l\in Hom_{Ab}(A, B[[z^B_1,\cdots, z^B_l]])$ 
for $l\ge 0$ which satisfy the following properties:

\begin{enumerate}[i)]
\item for any $\alpha\in A$ the series $G_l(\alpha)$ is divisible by $z_1^B\cdots z_l^B$;
\item for any $\alpha\in A$ the series $G_l(\alpha)$ is symmetric;
\item for any $\alpha\in A$ the following system of equations is satisfied
\begin{equation}\label{eq:add}
G_l(\alpha)(z_1^B, z^B_2, \ldots, z^B_{l-1}, F_B(x,y)) = 
\sum_{i,j} G_{l+i+j-1}(\alpha a_{i,j})(z_1^B, z_2^B, \ldots, z^B_{l-1}, x^{\times i}, y^{\times j})
\end{equation}
\end{enumerate}

Here $x^{\times i}$ and $y^{\times j}$ denote $i$-tuple $(x,x,\ldots, x)$ and $j$-tuple
$(y,y,\ldots, y)$ respectively.
\end{BigTh-add}

Note that i) is an instance of continuity discussed in Section \ref{sec:cont}.

\begin{Rk}\label{rem:add_grad}
If theories $A^*$ and $B^*$ are graded, 
one can use Remark \ref{rem:grad} to specify the data of additive operations between graded components.
Additive operations from $A^n$ to $B^m$ are in 1-to-1 correspondence
with maps $G_l\in Hom_{Ab}(A^{n-l},B[[z_1,\ldots, z_l]]_{(m)})$
satisfying properties i)-iii) of the CAOT.
\end{Rk}

\subsection{Derivatives and products of poly-operations}\label{subsec_poly}

There are two straight-forward ways to produce some poly-operations from operations,
or in other words to increase the {\sl arity} of operations.

First, if $\phi_1, \phi_2$ are $r_1$-ary and $r_2$-ary poly-operations, respectively,
then we define an $(r_1+r_2)$-ary poly-operation $\phi_1\odot \phi_2$ 
as their external product: 
$$(\phi_1\odot \phi_2) (x_1,x_2,\ldots,x_{r_1},y_1,y_2,\ldots,y_{r_2}) = \phi_1(x_1,x_2,\ldots,x_{r_1})\phi_2(y_1,y_2,\ldots,y_{r_2}).$$

This construction defines a morphism of algebras
$$ [(A^*)^{\times r_1},B^*\circ \prod^{r_1}]\ot_B [(A^*)^{\times r_2},B^*\circ\prod^{r_2}] \xrarr{\odot} 
[(A^*)^{\times (r_1+r_2)}, B^*\circ \prod^{r_1+r_2}].$$

One may interprete the statement that the latter map is an isomorphism 
for particular $A^*$ and $B^*$ as some kind of Kunneth-type property.
When this property is satisfied for all $r_1, r_2$ 
(cf. Th. \ref{chern_polybase}, Prop. \ref{ch-base} and \ref{polyderint}),
we will write $[(A^*)^{\times r}, B^*\circ \prod^r] = [A^*,B^*]^{\odot r}$.

Second, if $\phi$ is an $r$-ary poly-operation,
then we define an $(r+1)$-ary poly-operation $\partial^1_i \phi$ 
as its derivative with respect to the $i$-th component (\cite[Def.3.1]{PS_Vish2}).
Denote by $Z_{<i}=(z_1,\ldots,z_{i-1})$, $Z_{>i}=(z_{i+1}, \ldots, z_r)$, then 
$$ \partial_i \phi(Z_{<i},x,y,Z_{>i}) := 
\phi(Z_{<i},x+y,Z_{>i})
-\phi(Z_{<i},x,Z_{>i})-\phi(Z_{<i},y,Z_{>i}).$$

It is clear that $\phi$ is poly-additive if and only if $\partial_i \phi = 0$ for $1\le i \le r$.

If $r=1$, i.e. $\phi$ is an operation, we will omit the subscript
and write $\partial \phi$ to mean its derivative.
Iterating the procedure one can easily define
$\partial^s_{(r_1,\ldots,r_s)}=\partial_{r_s}\circ\partial_{r_{s-1}}\circ \cdots \circ \partial_{r_1}$.
However, it is easy to see that all $s$-derivatives of an operation are symmetric
and thus derivatives do not depend on the order of derivation.
We will write $\partial^s \phi$ to denote any of them.

By definition of the derivative of $\phi$ one can express
values of $\phi$ on the sum of {\it two} elements
as the sum of values of $\phi$ and $\partial^1\phi$.
It is useful for computations to have analogous formulas
for the values on the sum of any number of elements.

\begin{Prop}[Discrete Taylor Expansion, Vishik, {\cite[Prop. 3.2]{PS_Vish2}}]\label{prop:taylor}
Let $f:A\rarr B$ be a map between abelian groups.
Denote by $\partial^if:A^{\times i}\rarr B$ its derivatives.

For any set $\{a_i\}_{i\in I}$ of elements in $A$ the following equality holds:

$$ f(\sum_{i\in I} a_i) = \sum_{J\subset I} \partial^{|J|-1}f(a_j|j\in J). $$
\end{Prop}

\section{Operations from $K_0$ to orientable theories}\label{sec_op_K0}

\subsection{Chern classes as free generators of operations from $K_0$}

The following Theorem was communicated to the author by A.Vishik.
The result is analogous to the usual calculation of generalized cohomology of a product of infinite Grassmannians, though it does not formally follow from it.

\begin{Th}[Vishik]\label{basis}
Let $A^*$ be a g.o.c.t.
Then the ring of $r$-ary poly-operations from a presheaf $\tilde{K_0}$ to $A^*$ 
is freely generated over $A$ by external products of Chern classes.

Using notations from section \ref{subsec_poly}, 
we write  $[(\tilde{K}_0)^{\times r}, A^*\circ \prod^r] = A[[c_1^A,\ldots,c_i^A,\ldots]]^{\odot r}$.
\end{Th}
\begin{Rk}
Note that there is no issue of convergence of a series of Chern classes for any particular element of $K_0$.
Due to a discussion after Def. \ref{goct}, Chern classes $c_i^A$ are nilpotent
 and thus a formal power series reduces to a polynomial for each particular variety.
\end{Rk}

\begin{proof}
For simplicity we will assume that $\phi$ is an operation (i.e. 1-ary poly-operation). 
In the end of the proof we explain how to generalize it for arbitrary arity.
We can also assume that $\phi(0)=0$ subtracting the constant operation if needed.

By the COT the operation $\phi$ is determined
by its restriction to products of infinite-dimensional projective spaces. 
Using the direct limit along inclusions as explained in Section \ref{sec:cont},
 $\phi$ is determined by a unique map of sets
$\phi:\ZZ[[z_1, z_2,\ldots, z_i, \ldots ]]\rarr A[[t_1, t_2, \ldots, t_i,\ldots]]$,
where $z_i = z_i^{K_0}$ and $t_i=z_i^A$.
This map has to commute with pull-backs along morphisms between finite products of projective spaces.
These pull-backs induce endomorphisms of the ring $\ZZ[[z_1, z_2,\ldots, z_i, \ldots ]]$
acting as identity on all but finite number of variables. 

Denote by $Q$ the symmetric series $\phi(z_1+z_2+\ldots+z_i+\ldots)\in A[[t_1,\ldots,t_i,\ldots]]$.

Let us prove the following claims from which the theorem follows:
\begin{enumerate}[i)]
\item for any symmetric $Q$ there exist a unique $A$-series of Chern classes $P(c_1,\ldots, c_i, \ldots)$\\
s.t. ${P(c_1,\ldots) (z_1+z_2+\ldots+z_i+\ldots) = Q(t_1,\ldots, t_i, \ldots)}$;

\item if $Q=0$, then the operation $\phi$ is zero itself.
\end{enumerate}

i) Due to the fundamental theorem of symmetric series 
$Q$ can be written as a power series $P$ over $A$ in elementary symmetric series in
$t_i$’s.
However, these elementary symmetric series are exactly
 values of Chern classes $c_i^A(z_1 + z_2 +\ldots)$. 
Starting from $Q$ we have constructed $P$ satisfying the given condition.

ii) Suppose that $Q=0$, and let us show that $\phi$ is zero in two steps. 
First, we will show that any
element in $\ZZ[z_1,\ldots, z_l]\subset \ZZ[[z_1,\ldots, z_l, \ldots]]$ (for all $l>0$) 
without a constant term
can be given as 
a pull-back of $q:=z_1-z_2+z_3-z_4+\ldots+z_{2i-1}-z_{2i}+\ldots$ along some morhism
 between projective spaces. 
As $\phi$ is supported on $\tilde{K}_0$ and
 by the continuity of operations, the value of $\phi$ on polynomials without constant term
 uniquely determines it, and thus $\phi$ is uniquely determined by its value on $q$.
Second, we will show that if $Q=0$, then $\phi(q)=0$.

{\bf Step 1}.
In fact, we will need only those morphisms which appear in the statement of the COT.
Let us write explicitly how the pull-backs along them look like (cf. \cite[5.1]{PS_Vish2}).
Below we consider the maps 
$Id^{\times (r-1)}\times f$ from the 
product of $\mathbb{P}^\infty$ where $f$ is either the diagonal map, the Segre map
or the point embedding. Clearly, 
the correspoding pull-backs act identically on almost all variables $z_i$,
the exceptional cases being writted down explicitly:
\begin{itemize}
\item the partial diagonal: $z_{r+1}\rarr z_{r}$;
\item the partial Segre map $Seg_r: z_r \rarr z_r + z_{r+1} + z_rz_{r+1}$
\item the partial point embeddings: $z_{r+1} \rarr 0$;
\end{itemize}

We also will need to use the action of the symmetric group $\cup_{n=1}^\infty S_n$ 
which acts by permutations on $z_i$.

It is enough to prove that for any $r>0$
we can get monomials $\pm z_1\cdots z_r$ from $q$ using transformations above.
Setting some of the variables to be equal (i.e. using pullbacks along partial diagonals)
 we will thus obtain the value on any monomial of degree $r$ 
 (with plus and minus sign).
As we can apply transformations to different groups of variables independently
and as we can get any monomial, we can obtain any polynomial as well.

Let us prove that $\pm z_1\cdots z_r$ can be obtained from $q$ 
by induction on $r$.

{\bf Base of induction ($r=1$).} 
Setting zero all the variables (i.e. using pull-backs along point inclusions)
 except for $z_1$ or $z_2$ we get either $z_1$ or $-z_2$,
respectively. Using the permuation $(12)$ one gets $-z_1$ from $-z_2$.

{\bf Induction step.}
Applying composition of Segre transformations $Seg_r\circ Seg_{r-1}\circ \cdots \circ Seg_1$ 
to variable $z_1$ we get
$\sigma_1 + \sigma_2 + \ldots + \sigma_{r+1}$ where $\sigma_i$ 
is the elementary symmetric function of degree $i$ in $z_1, \ldots, z_{r+1}$.
Starting from $-z_1$ we would get the same expression with the minus sign.

By induction assumption and using other groups of variables in $q$ we may get
polynomials $-z_{i_1}\cdots z_{i_{l-1}}$ for any $l: 1\le l\le r$  and any set of indexes $i_j$.
Thus, we may cancel all monomials of degree less than $r+1$ in 
the expression $\sigma_1 + \sigma_2 + \ldots + \sigma_{r+1}$
and the induction step is proved.

The proved claim shows that the value of $\phi$ on 
the element $q$ determines the operation uniquely.

{\bf Step 2}. Assume that $Q=0$. Let us show that $\phi(q)=\phi(z_1+z_3+z_5+\ldots-z_2 -z_4-z_6-\ldots)=0$.

By continuity (Section \ref{sec:cont}), it is enough to show that 
$\phi(z_1+z_2+z_3+\ldots+z_n -z_{n+1} -z_{n+2}-\ldots-z_{2n})=0$ for any $n\ge 1$.
However, by the Discrete Taylor Expansion as stated in Prop. \ref{prop:taylor}
this would follow from $\partial^i\phi(\pm z_1, \pm z_2,\ldots, \pm z_{i+1})=0$ for any $i\ge 0$.
Let us prove the latter claim 
by induction on the number of minuses in $(\pm z_1, \pm z_2,\ldots, \pm z_{i+1})$.

{\bf Base of induction (no minuses).} 
By the definition of the derivative (cf. Discrete Taylor Expansion)
we may express $\partial^i\phi(z_1,\ldots,z_{i+1})$ 
as the sum $\sum_{I\subset \{1,\ldots,i+1\}}(-1)^{i+1-|I|}\phi(\sum_{j\in I} z_j)$
which is zero by assumption.

{\bf Induction step}. Assume that the claim is true for at most $k$ minuses.

Note that as $Q$ is zero, $\phi(0)=0$.
It is easy to prove by induction (unrelated to the induction on minuses) that
 $\partial^s \phi(P_1, \ldots, P_{s+1}) =0 $ whenever $P_i=0$ 
for some $i:1\le i \le s+1$.

Let $\mathbf{z}_{i}= (\pm z_1, \pm z_2,\ldots, \pm z_{i})$ have at most $k$ minuses.
We will show now that 
\begin{equation}\label{eq2}
\partial^i\phi(\mathbf{z}_{i}, -z_{i+1}) =
 (-1)^n \partial^{i+n} \phi (\mathbf{z}_{i}, (z_{i+1})^{\times n}, -z_{i+1}),
\end{equation}
for any $n\ge 0$.

Using the identity $0 = z_{i+1} - z_{i+1}$ 
we get that 
\begin{multline}
0=
\partial^{i+n} \phi (\mathbf{z}_{i}, (z_{i+1})^{\times n}, z_{i+1}-z_{i+1})
= \partial^{i+n} \phi (\mathbf{z}_{i}, (z_{i+1})^{\times (n+1)}) +\\
+\partial^{i+n} \phi (\mathbf{z}_{i}, (z_{i+1})^{\times n}, -z_{i+1})
+ \partial^{i+n+1} \phi (\mathbf{z}_{i}, (z_{i+1})^{\times n}, -z_{i+1}).
\end{multline}

The formula \ref{eq2} follows by induction
 since the first summand in the RHS is zero by our assumptions.

Now note that $\partial^i \phi(\pm z_1,\ldots,\pm z_{i+1})$ 
is divisible by $t_1\cdots t_i$ as a series over $A$
(this is again an instance of continuity of operations).
Indeed, setting any of the variables $z_i$ to zero 
(that is, restricting along the respective partial
point embedding)
has to lead to the annihilation of the value.

Restricting along the diagonal $\mathbb{P}^\infty\rarr(\mathbb{P}^\infty)^{\times n+2}$
 on the last $(n + 2)$ factors (i.e, setting the respective variables $z_j$ to be equal),
we obtain that $\partial^{i+n} \phi (\mathbf{z}_{i},z_{i+1}^{\times n}, -z_{i+1})$
has to be divisible by $t_{i+1}^{n+2}$ and the same
should hold for $\partial^i\phi(\mathbf{z}_{i}, -z_{i+1})$. 
As this is true for all $n$ we
get that $\partial^i\phi(\mathbf{z}_{i}, -z_{i+1})=\partial^i\phi(\mathbf{z}_{i}, 0)=0$.
This proves the induction step and the theorem for the case of operations.\\

To generalize this proof to a poly-operation $\psi$
one needs to prove i) and ii) for symmetric series $Q$
defined as the value of $\psi$ on 
$z^{(1)}_1+z^{(1)}_2+z^{(1)}_3+\ldots,z^{(2)}_1+z^{(2)}_2+\ldots, \ldots, z^{(r)}_1+z^{(r)}_2+\ldots$.
 Part ii) is nearly the same,
while in part i) one needs to apply the fundamental theorem of symmetric series
for each set of variables $t^{(j)}_1,t^{(j)}_2, \ldots$
\end{proof}

\subsection{Additive operations from $K_0$ to an oriented theory $A^*$}
In this section
we deduce the description of all additive operations $[\tilde{K}_0,A^*]^{add}$
in terms of series in Chern classes. 
This result is not used anywhere in the paper, 
however it provides some intution about what one could expect 
of additive operations from higher Morava K-theories to Chow groups $[K(n)^*,CH^*\ot\Zp]^{add}$ 
(compare Cor. \ref{cr_k0_add} and Cor. \ref{cr_kn_add}).

Denote by $P_n\in \QQ[c_1,\ldots,c_n]$ a polynomial, 
which is the $n$-th graded component of $\log (1 + c_1 + c_2 + \ldots + c_n)$ multiplied by $n$ 
(variable $c_i$ has degree $i$ here). 
For example, $P_1 = c_1, P_2 = 2c_2+(c_1)^2, P_3 = 3c_3+3c_1c_2+(c_1)^3$.

\begin{Prop}
\begin{enumerate}
\item $P_n\in \ZZ[c_1,\ldots,c_n]$;
\item $P_n$'s produce linearly independent (over $A$) additive operations from $K_0$ to $A^*$;
\item all additive operations from $\tilde{K}_0$ to $A^*$ are 
infinite $A$-linear combinations of $P_n$'s.
\end{enumerate}
\end{Prop}

As was noted previously, Chern classes are nilpotent,
thus there is no issue of convergence of an infinite linear combination of $P_n$'s
for values on each particular variety.

\begin{proof}

1. Consider $f(t)=\log (1+ c_1t +c_2t^2+c_3t^3+\ldots)$ as a polynomial in $t$ 
over the ring $\QQ[c_1,\ldots, c_i,\ldots]$.
 We are interested in the coefficient of $t^n$ in $f$ as a polynomial of $c_i$'s.

Note that $f'(t) = \log'(1+c_1t+c_2t^2+\ldots) (c_1+2c_2t+3c_3t^2+\ldots)$ is an integral series,
because $\log'(1+x)$ is integral. As the coefficient of $t^n$ is multiplied by $n$
after differentiation the claim follows.

It is easy to see that $P_n$ contains $(c_1)^n$ as a summand,
so the common divisor of coefficients in $P_n$ equals to 1.

2. It follows from the Cartan's formula that $P_n$'s produce additive operations
from $K_0$ to $A^*$.
Any $A$-linear relation will give a non-trivial relation between Chern classes 
which would contradict Theorem \ref{basis}.

3. From the CAOT it follows that
any additive operation from $\tilde{K}_0$ to $A^*$ is determined 
by symmetric polynomials $G_l\in A[[z_1,\ldots, z_l]] \cdot \prod_{i=1}^l z_i$ for each $l\ge 1$
which  satisfy the following equations:

$$ G_l(z_1,z_2, \ldots, z_{l-1}, F_A(x,y)) = G_l(z_1,z_2,\ldots, z_{l-1}, x) + G_l(z_1,z_2,\ldots, z_{l-1}, y)
 + G_{l+1}(z_1, z_2, \ldots,z_{l-1},x,y).$$

From this equation one may express $G_{l+1}$
in terms of $G_l$, and by induction the operation is uniquely determined by the series $G_1$.
We will show now that for the operation $P_n$ the corresponding series $G_1$ is just $z^n$.
Therefore any series $G_1\in A[[z]]\cdot z$ can be realized as 
the value of some infinite linear combination of $P_n$'s.

By definition $G_1$  equals to $\phi(c_1^{K_0}(\mathcal{O}(1)))$ expressed 
as a series over $A$ in $z:=c_1^{A}(\mathcal{O}(1))$. 
As $c_1^{K_0}(\mathcal{O}(1))=\mathcal{O}(1)-\mathcal{O}$,
it is easy to see that for $c^A_1$ the series $G_1$ equals to $z$ 
and for $c^A_i$, $i>1$ it equals to 0.
Therefore for operations $P_n$ the series $G_1$ is equal to $z^n$. 
\end{proof}

\begin{Rk}
This proposition is a generalization of a result of A.Vishik (\cite[Th. 6.8]{PS_Vish1})
 classifiying additive operations from $K_0$ to itself. 
Vishik has shown that any additive operation
 is a unique infinite linear combination
of operations $\Upsilon_k:=\sum_{i=1}^k(-1)^{i-1}\binom{k}{i}\psi_{i}$,
where $\psi_i$ are Adams operations.
The proofs of Vishik's result and of ours are quite similar
and one can check that $\Upsilon_k$ equals to a polynomial $P_k$ on Chern classes $c_i^{K_0}$. 
\end{Rk}

\begin{Cr}\label{cr_k0_add}
The natural map $[K_0, CH^i\ot\Zp]^{add}/p\rarr [K_0, CH^i/p]^{add}$ is
an isomorphism of 1-dimensional vector spaces.

In particular, take $\phi_i, \phi_{ip}$ to be two generators of
additive integral operations from $K_0$ to $CH^i\ot\Zp$, $CH^{ip}\ot\Zp$ respectively.
Then $(\phi_i)^p \equiv a\phi_{ip} \mod p$ for some $a\in\F{p}$.
\end{Cr}

\section{Localisation of theories and Adams operations}\label{sec_adams}

\subsection{Adams operations are universally central.}

The following proposition was communicated to the author by A. Vishik.

\begin{Prop}[Vishik, cf. {\cite[Th. 6.15]{PS_Vish1}}]
Let $A^*$ be a theory of rational type. 
Then there exist $A$-linear multiplicative Adams operations $\psi^A_i:A^*\rarr A^*$ for $i\in \ZZ$,
which are uniquely defined by the property $\psi^A_i(c_1^A(L))=c_1^A(L^{\ot i})$,
where $L$ is any line bundle over any smooth variety.

Moreover, Adams operations do not depend on the orientation.
\end{Prop}
\begin{proof}
The existence of Adams operations was proved in \cite{PS_Vish1}
and uniqueness follows from the COT.

Using the fact that Adams operations
are $A$-linear (i.e., that $\psi_i$ acts identically on $A$),
 it is easy to see that they are stable under reorientation.
\end{proof}

\begin{Prop}\label{adams}
Let $\phi:A^*\rarr B^*$ be an operation between two theories of rational type.

Then $\phi$ commutes with Adams operations, i.e. $\phi \circ \psi_k^A = \psi_k^B \circ \phi$.
\end{Prop}
\begin{proof}
According to the COT we may check the 
equality $\phi \circ \psi_k^A = \psi_k^B \circ \phi$ on products of projective spaces.
However, the action of Adams operations $\psi_k$ on products of projective spaces $(\mathbb{P}^\infty)^r$
is nothing more than the pull-back along the composition of Segre maps with diagonals 
(i.e. products of $k$-Veronese maps) and, thus, commutes with any operation.
\end{proof}

\subsection{Localisation of non-additive operations.}

\begin{Prop}\label{prolong}
Let $S$ be a subset in $\mathbb{Z}\setminus\{0\}$, denote by $\ZZ_S:=S^{-1}\ZZ$ the localisation
of integers in $S$.

Let $A^*$ be a theory of rational type s.t. the map $A\rarr A\ot\ZZ_S$ is injective,
and let
$B^*$ be any g.o.c.t such that $S$ is invertible in $B$.

Then the natural map $[\tilde{A}^*\ot\ZZ_S, B^*] \rarr [\tilde{A}^*, B^*]$ is an isomorphism.
\end{Prop}

This Proposition is obviously true for additive operations, though to deal with non-additive ones
we use the COT. 

\begin{proof}
From the COT it follows that any operation from $A^*$ to $B^*$
factors through a theory of rational type $\Omega^*\ot_{\mathbb{L}} B$.
Thus, without loss of generality we may assume that $B^*$ is of rational type as well. 

Let $\phi$ be any operation from $\tilde{A}^*$ to $B^*$.
We need to show that there exists unique $\bar{\phi}:\tilde{A}^*\ot\ZZ_S\rarr B^*$,
s.t. its composition with the natural map 
$\tilde{A}^* \rarr \tilde{A}^*\ot\ZZ_S$ is equal to $\phi$.

Denote by $A_{r,n}:=\tilde{A}^*((\mathbb{P}^n)^{\times r})$ a factor ring of $A[[z_1,\ldots, z_r]]$
by the ideal of power series of degree $\ge(n+1)$, 
where $z_i$ denotes the first Chern class
of the line bundle $\mathcal{O}(1)_i$ (cf. section \ref{sec:cont}).
By the COT and continuity of operations $\bar{\phi}$ 
is determined by its restriction to maps of sets from $A_{r,n}\ot\ZZ_S$ to $B_{r,n}$ for all $r,n$.

We claim that for any $P\in A_{r,n}\ot \ZZ_S$ there exists $M\in \ZZ_S^\times$, s.t.
 $\psi^A_M (P) \in A_{r,n}$.
Recall that the 
operation $\psi^A_M$ is multiplicative and $\psi^A_M(z)=Mz+ Q(z)$ where $Q\in A[[z]]\cdot z^2$.
Write $P=P_{< k}+P_k + P_{>k}$ where summands are of degree less than $k$, exactly $k$ and bigger than $k$,
respectively. Assume that summands of degree less than $k$ have coefficients from $A$,
i.e. $P_{< k}\in A_{r,n}$.
Let $d$ be the common denominator of coefficients of $P_k$.
Apply $\psi^A_d$ to $P$. The polynomial $\psi^A_d(P_{< k})$ still has coefficients in $A$,
the polynomial $\psi^A_d(P_{>k})$ still has degree bigger than $k$, 
and at the same time the polynomial $\psi^A_d(P_{k})$ has its coefficients in degree $k$ 
multiplied by $d^k$ 
and thus lying in $A$.
Therefore $\psi^A_{d}(P)$ has its summands of degree less than $k+1$ lying in $A$.
Iterating this procedure and using the fact that $\psi^A_{m}\circ \psi^A_{n}=\psi^A_{mn}$, 
the claim is proved. 

By Prop. \ref{adams} Adams operations commute with any operation, 
$\bar{\phi}(\psi^A_M (P)) = \psi^B_M\circ \bar{\phi}(P)$. 
On the other hand $\psi^B_M$ is an invertible operation when $M$ is invertible in $B$.
Therefore this equality allows to express $\bar{\phi}(P)$ in terms of $\phi$ uniquely,
which proves the uniqueness of $\bar{\phi}$. 

One can define $\bar{\phi}$ by the procedure above.
It is enough to show that the maps $\bar{\phi}:A_{r,n}\rarr B_{r,n}$ 
commute with the restriction maps of the list of morphisms in the COT
(Segre, diagonals, etc.).
However, this follows quite formally as $\phi$ and Adams operations commute 
with pull-backs along all morphisms between projective spaces.
\end{proof}

\begin{Prop}\label{ch-base}
Let $K$ be a theory of rational type with the ring of coefficients being a subring in $\mathbb{Q}$.

Then any $r$-ary poly-operation from $\tilde{K}$ to $CH^*\ot\QQ$ 
can be uniquely written as a series in external products of monomials in $\{ch_i\}_{i\ge 1}$,
where $ch:K\rarr CH^*\ot\QQ$ is the unique stable invertible multiplicative operation 
aka Chern character (cf. \cite[Th. 3.7, Prop. 3.8]{PS_Vish1}).

In other words, $\left[\tilde{K}^{\times r},CH^*_{\QQ}\circ \prod^r\right] 
= \QQ[[ch_1,\ldots, ch_i,\ldots]]^{\odot r}$.
\end{Prop}
\begin{proof}
By Prop. \ref{prolong} the natural map
 $[\tilde{K}\ot\QQ, CH^*\ot\QQ]\rarr [\tilde{K}, CH^*\ot\QQ]$ is an isomorphism. 
As any two formal group laws over $\QQ$-algebras are isomorphic,
there exists the unique stable invertible multiplicative operation
yielding an isomorphism
$\tilde{K}_0\ot\QQ \cong \tilde{K}\ot\QQ$, 
which, however, does not respect the push-forward structure.
Obviously, it maps components of the Chern character
 $K\ot\QQ \rarr CH^*\ot \QQ $ to components of 
the Chern character $K_0\ot \QQ \rarr CH^*\ot\QQ$. 
Thus, it is enough to prove the statement for $K=K_0$.

Note that there is a standard equality $(\log (1+c_{tot}))_n = (n-1)!ch_n$ 
between operations from $K_0$ to $CH^n\ot\QQ$,
where $\log (1+x) = x+\frac{1}{2}x^2+\frac{1}{3}x^3+\ldots$ 
and $c_{tot}=c_1+c_2+c_3+\ldots$ is the total Chern class.
Using it one can uniquely express Chern classes as polynomials
in the components of the Chern character $ch_n$ and vice versa.
The claim now follows from Theorem \ref{basis}.
\end{proof}

\begin{Cr}\label{cor-dim_op}
The dimension of the space of operations $[K, CH^i\ot \QQ]$ equals to $p(i)$,
where $p(i)$ is the number of partitions of $i$.

The dimension of the space of $r$-ary poly-operations from $K$ to $CH^i\ot\QQ$ does not depend on $K$
(and can be interpreted as a number of some partition-type objects).
\end{Cr}

\subsection{Applications: theories of rational type with the ring of coefficients equal to $\Zp$.}

\begin{Prop}\label{add_rank}
Let $A^*$ be a theory of rational type s.t. $A=\Zp$ and let $i\ge 1$.

Then the $\Zp$-module of additive operations $[\tilde{A}^*, CH^i\ot \Zp]^{add}$ is free of rank 1;
the $\Zp$-module of all operations $[\tilde{A}^*, CH^i\ot \Zp]$ is free of rank $p(i)$,
	where $p(i)$ is the number of partitions of $i$.
\end{Prop}
\begin{proof}
The $\Zp$-modules $[\tilde{A}^*, CH^i\ot \Zp]^{add}$ and $[\tilde{A}^*, CH^i\ot \Zp]$
are torsion-free and are submodules of operations to $CH^i\ot\QQ$
as follows from the COT.
The spaces of operations to $CH^i\ot\QQ$ are finite dimensional, 
thus modules of operations to $CH^i\ot\Zp$ are of finite rank and free.

We claim that the natural map $[A^*,CH^*\ot\Zp]\ot\QQ \rarr [A^*,CH^*\ot\QQ]$ is an isomorphism.
This would prove the proposition as ranks of $\QQ$-modules of operations 
were computed in Cor. \ref{cor-dim_op}.
One may reformulate the claim as follows: for any operation $\psi:A^*\rarr CH^i\ot\QQ$ 
there exists $n$ s.t. $p^n\psi$ factors through $CH^i\ot\Zp$.

By the COT operation $\psi$
is specified by its value on products of projective space,
i.e. on $A^*((\mathbb{P}^{\infty})^{\times r})=\Zp[[z_1,\ldots, z_r]]$
where $z_j = c_1^A(\mathcal{O}(1)_j)$.
 Denote also by $t_j$'s the first Chern classes $c_1^{CH}(\mathcal{O}(1)_j)$ 
 in the target theory.

First, let us show that $\partial^{i+1} \psi = 0$.
From the definition of derivatives it follows that 
$\partial^s \psi(P_1,\ldots,P_{s+1})=0$ whenever $P_i=0$ for some $i$.
It is enough to show that $\partial^{i+n}\psi(P_1,\ldots,P_{i+n+1})=0$ for all $n\ge 1$
and all {\it monomials} $P_j=az_1\cdots z_s$ since values on any tuple of polynomials
can be expressed by higher derivatives of values on such monomials 
using  Discrete Taylor Expansion (Prop. \ref{prop:taylor}).
 Denote by $a_k\bar{z}^{(k)}=a_kz^{(k)}_1\cdot z^{(k)}_2\cdots z^{(k)}_{r_k}$
for some $r_k\ge 1$ and $a_k\in\Zp$.
The value $\partial^{i+n}\psi(a_1\bar{z}^{(1)},a_2\bar{z}^{(2)},\ldots,a_{i+n+1}\bar{z}^{(i+n+1)})$
has to be divisible by $\prod_{k=1}^{i+n+1} \prod_{j=1}^{r_k} t^{(k)}_{j}$ which has degree 
bigger than $i$ whenever $n>0$,
and thus is zero.

Second, now it is easy to see that the operation 
$\psi$ is specified by a finite number of values
$\partial^n \psi(\bar{z}^{(1)},\bar{z}^{(2)},\ldots$\\$\ldots,\bar{z}^{(n-1)})$
for all $r_k\le i$ and $n\le i$. 
Indeed, values $\partial^n \psi(a_1\bar{z}^{(1)},a_2\bar{z}^{(2)},\ldots,a_{n-1}\bar{z}^{(n-1)})$
for $a_i\in \NN$ can be expressed in terms of values above
by the definition of derivatives, and for $a_i\in\Zp$
one needs to use invertible Adams operations in the same way as in
the proof of Prop. \ref{prolong}.
Therefore there exist $N$ 
s.t. for the operation $p^N\psi$ all these polynomials will be $p$-integral 
and therefore the operation will be $p$-integral as well.
\end{proof}

\section{Operations from Morava K-theories to Chow groups}\label{sec_op_mor}

Fix a prime $p$. All the theories considered in this section will have
a structure of $\Zp$-algebras 
and the adjective {\sl integral} will always mean {\sl defined over $\Zp$} (aka $p$-integral).

\subsection{Definitions and grading}\label{subsec_mor}

\begin{Def}\label{mor} Let $n\ge 1$.
A theory of rational type with the ring of coefficients $\Zp$
is called $n$-th Morava K-theory 
if the logarithm of the corresponding formal group law $F_{K(n)}$
has the form $\sum_{i=0}^\infty b_i x^{p^{ni}}$ (where $b_i\in\QQ$),
and $F_{K(n)}\mod p$ has height $n$.
\end{Def}

\begin{Rk}
In \cite{PS_PetSem} {\sl the} $n$-th Morava K-theory
was defined as a theory 
of rational type with
the Lubin-Tate formal group law
defined by the logarithm $\sum_{i=0}^\infty \frac{1}{p^i} x^{p^{ni}}$.
This gives an example of an $n$-th Morava K-theory of Def. \ref{mor}.

However, one can show that
not all of $n$-th Morava K-theories as we define them
are multiplicatively isomorphic to each other 
(in particular, not all of them are multiplicatively isomorphic
to the one in loc.cit.).
In a future paper we will show 
that all $n$-th Morava K-theories
are isomorphic as presheaves of abelian groups, thus,
 after all there is the $n$-th Morava K-theory
 which might be endowed with different multiplicative and push-forward structures.
Perhaps, this justifies our 'ambiguous' definition.
\end{Rk}

\begin{Rk}
In topology one often defines the spectrum of the $n$-th Morava K-theory
as $(BP/(p, v_j, j\neq n))[v_n^{-1}]$,
and so its ring of coefficients is $\mathbb{F}_p[v_n, v_n^{-1}]$
where $\deg v_n=1-p^n$. The resulting spectrum does not depend on the choice of generators $v_j$.

 If, however, one chooses to consider $(BP/(v_j, j\neq n), v_n-a_1)$, where $a_1\in\Zp^\times$, 
similar to what we investigate, it is not clear whether the result is independent
of the choice of $v_j$. 
\end{Rk}

\begin{Rk}\label{rem_k0-mor}
The Artin-Hasse exponential $\exp(\log_{K(1)})$ establishes an isomorphism over $\Zp$
between the formal group law $F_m=x+y+xy$ and 
the formal group law with the logarithm $\log_{K(1)}(x)=\sum_{i=0}^\infty \frac{x^{p^i}}{p^i}$.

For corresponding g.o.c.t. this means
that K-theory $K_0\ot\Zp$ is isomorphic to a first
Morava K-theory as a presheaf of rings.
\end{Rk}

The formal group law of $K(n)^*$ looks as: 
$F_{K(n)}(x,y)=x+y-b_1\sum_{i=1}^{p^n-1} \binom{p^n}{i} x^iy^{p^n-i}+\textrm{higher terms}$,
where $a_1\in\Zp^{\times}$.

\begin{Prop}\label{prop:mor_grad}
\begin{enumerate}
\item The formal group law $F_{K(n)}=\sum_{i,j} a_{ij}x^iy^j$ of the theory $K(n)^*$
satisfies: $a_{ij}=0$ for $i+j\neq 1 \mod (p^n-1)$.

\item The Morava K-theory $K(n)^*$ is $\ZZ/(p^n-1)$-graded.

\item The grading on $K(n)^*$ is respected by Adams operations.

\item\label{item:pushforward} The grading is compatible with push-forwards,
i.e. for a proper morphism $f:X\rarr Y$ of codimension $c$
push-forward maps increase grading by $c$,
 $f_*:K(n)^i(X)\rarr K(n)^{i+c}(Y)$.
 
In particular, $c_1^{K(n)}(L)\in K(n)^1(X)$ for any line bundle $L$ over a smooth variety $X$.

\item Let $A$ be a $\Zp$-algebra containing the root
 of unity $\zeta$ of degree $(p^n-1)$. 
Then there exist a multiplicative operation $\hat{\zeta}$,
 which sends $c_1(L)$ to $\zeta c_1(L)$
for any line bundle $L$.
This operation acts as multiplication by $\zeta^j$ on $K(n)^j\ot A$.
\end{enumerate}
\end{Prop}
\begin{proof}
Let $\sum_{i=0}^\infty b_i x^{p^{ni}}$ be the logarithm of $F_{K(n)}$
and consider the FGL over the graded ring $\Zp[v_n,v^{-1}_n]$ (with $\deg v_n=1-p^n$)
 defined by its logarithm
$\sum_{i=0}^\infty b_iv_n^{\frac{p^{in}-1}{p^n-1}} x^{p^{in}}$.
Denote by $\hat{K}(n)^*$ the corresponding theory of rational type.

Since the logarithm is homogenous of degree 1 (with $x$ having degree 1),
the map $\mathbb{L}\rarr \Zp[v_n,v_n^{-1}]$
classifying the FGL preserves grading and the theory $\hat{K}(n)^*$
inherits the grading from algebraic cobordism. 
Moreover,
 $F_{\hat{K}(n)}=\sum_{i,j} \hat{a}_{ij}x^iy^j$ 
 has $\hat{a}_{ij}=0$ for $i+j\neq 1 \mod (p^n-1)$
for grading reasons.

The FGL $F_{K(n)}$ can be defined by setting $v_n=1$ in $F_{\hat{K}(n)}$
and the theory $K(n)^*$ is the localisation of $\hat{K}(n)^*$ along $v_n=1$.
Properties (1), (2), (4) now follow.

Note that the grading determines projectors $K(n)^*\rarr K(n)^i$
which have to commute with Adams operations by Prop. \ref{adams}. This proves property (3).

To prove (5) recall that multiplicative operations are in 1-to-1 correspondence
 with morphisms of FGL's (e.g. \cite[Th. 6.8]{PS_Vish1}). 
 Thus, we need to check that $\gamma(z)=\zeta z$ defines an endomorphism of the FGL $F_{K(n)}$.
This is clear from the property (1).
\end{proof}

We denote the graded components of $n$-th Morava K-theories
 as $K(n)^1$, $K(n)^2$, $\ldots$, $K(n)^{p^n-1}$,
and freely use the following expressions
 $K(n)^i$, $K(n)^{i \mod{p^n-1}}$, $K(n)^{i+r(p^n-1)}$
to denote the component $K(n)^j$ where $j\equiv i \mod p^n-1$, $1\le j\le p^n-1$.
The reason we denote the component $K(n)^{0}$ as $K(n)^{p^n-1}$
is mainly because we will work with $\tilde{K}(n)^*$
instead of $K(n)^*$, $\tilde{K}(n)^{p^n-1}$ 
contains classes of codimensions $p^n-1 + r(p^n-1)$ for all $r\ge 0$.

For Morava K-theories as well as for any free theory defined over the subring of $\QQ$,
 there exist a unique stable multiplicative operation $ch:K(n)^*\rarr CH^*\ot\QQ$
which we call the Chern character. 

\begin{Prop}\label{add_grad}
Let $\phi:K(n)^*\rarr CH^i\ot\Zp$ be any additive operation, $i\ge 1$.

Then it is supported on $\tilde{K}(n)^{i\mod p^n-1}$,
i.e. $\phi(x)=0$ for $x\in \Zp$ and for $x\in \tilde{K}(n)^j$ whenever $j\neq i\mod p^n-1$.
\end{Prop}
\begin{proof}
By the COT it suffices to check the statement for values of theories 
on products of projective spaces tensored by $\QQ$.
By Prop. \ref{add_rank} the operation $\phi\ot \id_\QQ$ is proportional
to $ch_i$ and therefore it suffices to prove that $ch_i$ maps $K(n)^j$ to zero
whenever $j\neq i \mod (p^n-1)$. Clearly $ch_i$ is zero on $\Zp\subset K(n)^*$.

From the projective bundle theorem
it follows that $K(n)^*((\mathbb{P}^\infty)^{\times r})\cong \Zp[[z_1,\ldots, z_r]]$,
{where} $z_i$ is equal to $c_1^{K(n)}(\mathcal{O}(1)_i)$. 
By Prop. \ref{prop:mor_grad}, (\ref{item:pushforward}) we have $z_i\in K(n)^1((\mathbb{P}^\infty)^{\times r})$,
and thus $\tilde{K}(n)^j((\mathbb{P}^\infty)^{\times r})$ is generated by monomials 
of degree $j+r(p^n-1)$, $r\ge 0$.

Let us show that 
\begin{equation}\label{eq:ch_grad}
ch(z)\in \oplus_{r=0}^\infty CH^{1+r(p^n-1)}(\mathbb{P}^\infty)\ot\QQ
\end{equation}

\sloppy{From the multiplicativity of the Chern character it would follow that} 
$ch(z_1\cdots z_i)$ lies in 
$\oplus_{r=0}^\infty CH^{i+r(p^n-1)}((\mathbb{P}^\infty)^{\times i})\ot\QQ$
 for any $i\ge 1$.

Note that the element $ch(z)$ can be expressed as a series in 
$t:=c_1^{CH}(\mathcal{O}(1))\in CH^1(\mathbb{P}^\infty)$
 which is inverse to the logarithm of the FGL $F_{K(n)}$.
As $\log_{K(n)}(x)$ contains only monomials in $x$ of degrees $1+r(p^n-1)$, $r\ge 0$,
the same is true for its inverse. The formula (\ref{eq:ch_grad})
 now follows and the proposition is proved.
\end{proof}

\subsection{Chern classes: statement of the main theorem}\label{subsec_mainth}

\begin{Th}\label{main}

There exist a series of non-additive operations $c_i:K(n)^*\rarr CH^i\ot\Zp$ for $i\ge 1$
satisfying the following conditions. 
\begin{enumerate}[i)]
\item The operation $c_i$ is supported on $K(n)^{i\mod (p^n-1)}$,
 i.e. $c_i(x)=0$ for $x\in K(n)^j$, where $j\neq i \mod p^n-1$.
\item Denote by $c_{tot} = \sum_{i\ge 1} c_i$ the total Chern class. Then the Cartan's formula holds: 
$$c_{tot}(x+y)=F^{K(n)}(c_{tot}(x),c_{tot}(y)).$$
\item Any operation $\tilde{K}(n)^*\rarr CH^*\ot\Zp$ can be 
written uniquely as a series in $\{c_j\}_{j\ge 1}$ over $\Zp$:
$$ [\tilde{K}(n)^*, CH^*\ot\Zp] = \Zp [[c_1,\ldots, c_i,\ldots ]].$$
\end{enumerate}
\end{Th}

\begin{Rk}\label{rem:petsem}
V.Petrov and N.Semenov introduced
 operations $c_1, c_2,\ldots,  c_{p^n}$ in \cite{PS_PetSem} 
 as additive operations to Chow groups
where summation was changed. 

They did not use the COT, which was not available at that time,
and the operations were constructed to $CH^*\ot\Zp$ modulo torsion.
\end{Rk}

\begin{Rk}
It follows from the theorem and the COT
that there are no relations between products of Chern classes modulo $p$,
i.e. $\F{p}[[c_1,\ldots,c_i,\ldots]]\subset [\tilde{K}(n), CH^*/p]$.

However, we will show in Section \ref{op_modp} that for $n>1$ there exist additive operations 
from $K(n)^*$ to $CH^*/p$
which are not liftable even as non-additive operations to $CH^*\ot\Zp$. 
This is different from the situation with the usual Chern classes from $K_0$ (i.e. the case $n=1$).
\end{Rk}

\begin{Rk}\label{rem_chern_n1}
By Remark \ref{rem_k0-mor} the theorem  is true for 
all theories $K(1)$ additively isomorphic to $K_0\otimes\Zp$.

Let us consider the case of $K(1)$ from Remark \ref{rem_k0-mor}.
By Prop. \ref{prolong} there exist Chern classes $c^{K_0}_i:K_0\ot\Zp\rarr CH^i\ot\Zp$,
satisfying the usual Cartan's formula 
which can be written as $c^{K_0}_{tot}(x+y)=F_m(c^{K_0}_{tot}(x),c^{K_0}_{tot}(y))$. 

Denote by $\chi:K(1)\rarr K_0\ot\Zp$ a stable invertible 
multiplicative operation corresponding to the morphism of formal group laws
$(\Zp, F_{K(1)})\rarr (\Zp, F_m)$ given by $(id, \gamma)$
where $\gamma$ is the inverse to the Artin-Hasse exponential.
Denote by $\bar{c}_{tot}=c^{K_0}_{tot}\circ \chi:K(1)\rarr CH^*\ot\Zp$ the composition of it 
with the usual total Chern class.
Since $\chi$ is multiplicative, and so, additive,
the usual Cartan's formula holds for $\bar{c}_{tot}$.

Define $c_i:K(1)\rarr CH^i\ot\Zp$ to be the $i$-th graded component of $\gamma(\bar{c}_{tot})$.
It satisfies property ii) of Theorem \ref{main} by construction and property iii) 
can be easily deduced from Theorem \ref{basis}. Property i) can be checked in the same manner
as it will be done for $K(n)^*$ in Section \ref{sec_cor}.
\end{Rk}

Briefly the proof of the theorem goes as follows.
We construct Chern classes by induction on degree. 
From the Cartan's formula it follows that the derivative $\partial^1 c_i$ of the operation $c_i$
is equal to a polynomial in Chern classes of smaller degree. 
To define $c_i$ one calculates an {\it anti-derivative} of $\partial^1 c_i$ 
as a $\QQ$-polynomial in $c_j$, $j<i$. The operation $c_i$
 may differ from this polynomial by any additive operation. 
Our goal is to find an additive operation $\psi_i$
s.t. its sum with the anti-derivative above is an integral operation, to be denoted $c_i$.
We reduce the problem of existence of $\psi_i$ to a certain question
 about additive operations from $K(n)^*$ to $CH^*/p$ (Lemma \ref{lift_add}).
This is done in Section \ref{chern_constr}. 

In Section \ref{op_modp} we investigate additive operations from $K(n)^*$ to $CH^*/p$.
Though we notice that there are many of them which are non-liftable to integral operations,
we find a sufficient condition for liftability as an additive operation that is sufficient 
for our purposes. 

It should be noted that in order to construct Chern classes 
we use only the CAOT  which appeared in \cite{PS_Vish1}. 
However, to show that constructed Chern classes form 
free generators of the ring of all operations we have to use poly-operations and the COT 
as stated in \cite{PS_Vish2}.

It easily follows from Prop. \ref{ch-base} 
that operations $c_i$ provide rational 
generators of the ring of all operations if the Chern character is expressible
as a rational series in these operations. The latter property will be satisfied by construction.
Thus, to prove iii) of Theorem \ref{main} it is enough to show that a non-integral polynomial in Chern classes
 will yield a non-integral operation. 
To do this we make a careful study of derivatives of poly-operations defined by polynomials in Chern classes finally
  reducing the question to poly-additive poly-operations.
  The latter problem  is quite easy. This is done in Section \ref{chern-base}. 
 In fact, we also show that external products of Chern classes provide free generators of poly-operations.

\subsection{Corollaries of the main theorem}\label{sec_cor}

In this section we deduce several corollaries of Theorem \ref{main}.
These corollaries are not applications of our result, but on the contrary they provide us a clue of how 
to construct Chern classes and prove the theorem. 
Lemma \ref{lm_add_powers} is proved unconditionally 
and will be used in subsequent sections as a tool, 
while Cor. \ref{cr_chern_exist} and \ref{cr_kn_add}
will be proved by the end of Section \ref{op_modp}.

{\sl Notation.}
Denote by $\nu_p(a)$ the $p$-valuation of a rational number $a\in \QQ$.
Let $Q$ be a polynomial with rational coefficients: 
$Q=\sum a_{(i_1,\ldots, i_n)} t_1^{i_1}\cdots t_n^{i_n}$.
Denote by $\nu_p(Q)$ the smallest $p$-valuation of coefficients of monomials of $Q$,
i.e. $\nu_p(Q) = \min \nu_p (a_{(i_1,\ldots, i_n)})$. 
Thus, $Q$ has coefficients in $\Zp$ if and only if $\nu_p(Q)\ge 0$. 
If $Q$ is $p$-integral, then it is divisible by $p$ if and only if $\nu_p(Q)>0$.

Let $P$ be a polynomial in variable $t$ over some ring $A$. Denote by $P[t^s]$
the coefficient of $t^s$ in $P$ (which is an element of $A$).\\

Define rational polynomials $P_i\in \QQ[c_1,\ldots,c_{i-1}]$ for each $i\ge 1$
 by the following formula:
$(\log_{K(n)} c_{tot})_i = c_i - P_i$. 
Here $(\log_{K(n)} c_{tot})_i$ is the $i$-th graded component of the series $\log_{K(n)} (c_1+c_2+\ldots)$
where $c_i$ has degree $i$.
For example, $P_1 = \ldots = P_{p^n-1} = 0$, $P_{p^n} = -\frac{a_1}{p}(c_1)^{p^n}$.

Note that from the Cartan's formula it follows 
that $(\log_{K(n)} c_{tot})_i$ 
defines an additive operation from $K(n)^*$ to $CH^i\ot\QQ$,
where $c_j$'s are operations from Theorem \ref{main}.
By Prop. \ref{add_rank}  this additive operation can be written as 
$p^{-\mu_i} \phi_i$, where $\mu_i\in \ZZ$ and $\phi_i$ is a generator of integral additive operations.

By iii) of Theorem \ref{main} the operation $\phi_i$ should be
 uniquely expressed as an integral polynomial in $c_1, \ldots, c_i$.
Moreover, this polynomial should not be zero modulo $p$ 
as $\phi_i$ is a generator of integral additive operations.
Combined with formulas above this gives us an equation $\phi_i = p^{\mu_i}c_i - p^{\mu_i}P_i$.
Thus, $\mu_i \ge 0$ (as $P_i$ does not contain $c_i$) 
and $p^{\mu_i}$ is divisible by the least common multiple of $p$-factors of denominators of $P_i$.
 
 More precisely, $\mu_i = \max(0, -\nu_p(P_i))$.  
 Indeed, if $P_i$ is integral (i.e. $\nu_p(P_i)\ge 0$),
  then the expression $c_i-P_i$ is already integral and not zero modulo $p$ 
  (as $P_i$ does not depend on $c_i$). 
 Otherwise, if $\nu_p(P_i)< 0$, we have $\mu_i = -\nu_p(P_i)$ 
 and the expression $\phi_i = p^{-\nu_p(P_i)}c_i - p^{-\nu_p(P_i)}P_i$ is integral
 and not zero modulo $p$.
 
Thus, we have proven the following Corollary,
which can be used to define operations $c_i$ inductively.

\begin{Cr}\label{cr_chern_exist}
Let $P_i$ be as above, and $\mu_i= \max(0, -\nu_p(P_i))$.

Then there exist a generator $\phi_i$ of additive operations $[\tilde{K}(n), CH^i\ot\Zp]^{add}$
s.t. the following equality between operations to $CH^i\ot\QQ$ holds
$$c_i = \frac{\phi_i}{p^{\mu_i}}+P_i(c_1,\ldots,c_{i-1}).$$
\end{Cr}

We will need the following Proposition (proved unconditionally)
to yield several polynomial relations between additive operations $\phi_i \mod p$.

\begin{Prop}\label{prop:log_morava}
Let $F$ be the FGL of an $n$-th Morava K-theory,
and denote the coefficients of its logarithm via the following formula:
$$ \log_{F}(x) = \sum_{i=0}^\infty b_i x^{p^{ni}},$$
where $b_i\in\QQ$, $i\ge 0$, $b_0=1$.

Then $b_i = \frac{a_i}{p^i}$ where $a_i\in\Zp^\times$,
and, moreover, $a_k \equiv (a_1)^k \mod p$ for $k\ge 1$.
\end{Prop}
\begin{proof}
By the result of Cartier the formal group law $F$
is $p$-typical (see e.g. \cite[I.§4]{PS_Araki}), 
and thus is classified 
by a map from the polynomial ring $BP=\Zp[v_1,v_2,\ldots]$ ([op.cit.]).
There exist the universal $p$-typical formal group law over $BP$
defined by the logarithm $\log_{BP} x=\sum_{i=0}^\infty l_ix^{p^i}$
where coefficients $l_j$ are expressed in terms of 
the generators $v_i$ by the Araki relations ([I.§6.4 (6.12), loc.cit.]):

$$pl_m = \sum_{i=0}^{m-1} l_i v_{m-i}^{p^i}+p^{p^m}l_m.$$

Clearly, the map $\pi:BP\rarr \Zp$ corresponding to the formal group law $F$
is determined by the fact that its rationalization $BP\ot\QQ\rarr \QQ$
sends $l_{ni}$ to $b_i$, and sends $l_j$ to zero whenever $n\nmid j$.

Recall also that for the universal $p$-typical 
formal group law we have 
$ [p]\cdot_{BP} x = px+_{BP}\sum^{BP}_{i\ge 1} v_ix^{p^i}$,
and therefore ${[p]\cdot_F x = px+_F \sum^F_{i\ge 1} \pi(v_i)x^{p^i}}$.
However, the $p$-series $[p]\cdot F=\log^{-1}_F (p\log F)$
 contains only monomials 
 which have degree in $x$ equal to 1 modulo $p^n-1$ 
 (let us call this 
property {\it gradability} throughout this proof).
In particular, it contains no terms of degree $i$ where $1<i<p^n$,
 and thus the morphism $\pi$ sends $v_i$ to zero for $i<n$.
The condition that $F \mod p$ has height $n$,
i.e. $[p]\cdot_F x \equiv ax^{p^n} \mod (p, x^{p^n+1})$
for some $a\in \F{p}^\times$,
means that $\pi(v_n)\in\Zp^\times$.

Let us show that $\pi$ sends $v_j$ to zero
whenever $n\nmid j$. Suppose that $i_0=\min\{j: \pi(v_j)\neq 0, n\nmid j\}$ is finite.
As mentioned above, we have 
$[p]\cdot_F x = px+_{F} \sum_i^{F} \pi(v_i)x^{p^i}$.
By our induction assumption it can be rewritten as
\begin{equation}\label{eq:p-sum-bpn}
[p]\cdot_{F} x = 
\sum_{jn<i_0}\nolimits^{F} \pi(v_{jn})x^{p^{jn}}+_{F}\pi(v_{i_0})x^{p^{i_0}}
+_{F}\mathrm{higher\ degree\ terms}.
\end{equation}

We claim that the RHS of (\ref{eq:p-sum-bpn}) 
has a summand $\pi(v_{i_0})x^{p^{i_0}}$ which contradicts 
the gradability of $[p]\cdot_F x$.
The first summand of the RHS of (\ref{eq:p-sum-bpn}) 
satisfies gradability up to degree bigger than $p^{i_0}$.
This follows from the fact that $F$ has summands only of bidegree $(i,j)$
where $i+j \equiv 1 \mod p^n-1$ (Prop. \ref{prop:mor_grad}), and the sum over $F$
of monomials having degree equal to 1 modulo $p^n-1$ clearly satisfies gradability.
The second summand of the RHS of (\ref{eq:p-sum-bpn}) 
yields $\pi(v_{i_0})x^{p^{i_0}}$, and all other summands in the RHS
have higher degrees in $x$, and therefore the claim is proved.

Now we can rewrite Araki relations pushed forward to $F$ 
and multiplied by a power of $p$ as 
\begin{equation}\label{eq:araki_modified}
p^{k}(1-p^{p^{nk}-1})b_{k} = \sum_{i=0}^{k-1} p^{k-1}b_{i}\pi(v_{(k-i)n})^{p^{in}},
\end{equation}
and use them to show that $p^kb_k\in\Zp^\times$,
and $p^kb_{k} \equiv (pb_1)^k \mod p$.

{\bf Base of induction ($k=1$).}  Denote by $a_1:=pb_1$.
From Araki relations (\ref{eq:araki_modified})
 it follows that $a_1\in\Zp^\times$
 and $a_1\equiv \pi(v_n) \mod p$.

{\bf Induction step.}
By induction we see that the RHS of Araki relations (\ref{eq:araki_modified}) is integral
and reducing modulo $p$ we get 
$$ p^kb_{k} \equiv p^{k-1}b_{(k-1)}\pi(v_n)^{p^{(k-1)n}} \equiv (a_1)^{k-1}a_1=(a_1)^k \mod p.$$
\end{proof}

Now we will calculate $\mu_i$ of Cor. \ref{cr_chern_exist} precisely,
 and this is done in the next Lemma which is proved unconditionally.
 Let us fix $n$-th Morava K-theory
for what follows 
with the logarithm 
$\log_{K(n)}(x) = x+\sum_{i\ge 1} \frac{a_i}{p^i}x^{p^{ni}}$
where $a_i \in \Zp^\times$ and $(a_1)^i\equiv a_i \mod p$
according to Prop. \ref{prop:log_morava}.

\begin{Lm}\label{lm_add_powers}

Let $C:=c_1 t+ c_2t^2 + \ldots = \sum_{i=1}^\infty c_it^i$  be a formal power series
in variables $c_i$, i.e. $C\in \ZZ[c_1,c_2,\ldots ][[t]]$. 
Define $P_s\in\QQ[c_1,c_2,\ldots, c_{s-1}]$
to be equal to $c_s-\log_{K(n)}(C)[t^s]$ for $s\ge 1$.

Fix $i>0$.

\begin{itemize}

\item[1.] Let $j$ be such that $j\neq i \mod (p^n-1)$. 
Set variables $c_s$ to be zero whenever $s \neq j \mod (p^n-1)$. 

Then the respective specialization $\tilde{P}_i$ of the polynomial $P_i$ is zero.

\item[2.] Let $v$ be such that $i=p^{nk}v$ and $p^n \nmid v$.
Then  $\nu_p(P_i)\ge -k$. 

\item[2p.] If moreover $k>0$, 
then $\nu_p(P_i)=-k$. 

Set variables $c_s$ to be zero whenever $s\neq i \mod (p^n-1)$.
Denote the respective specializations of the polynomials $P_i$ and $P_v$ 
by  $\tilde{P}_i$ and $\tilde{P}_v$. 

Then over $\F{p}[c_1,\ldots, c_i]$ polynomials 
$(p^k\tilde{P}_i \mod p)$ and $ ((c_v-\tilde{P}_v)^{p^{nk}} \mod p)$ are proportional.

\end{itemize}
\end{Lm}

\begin{Rk}
Perhaps, to calm oneself after the statement of this ad-hoc Lemma,
let us briefly explain how it will be used later.
In the course of an inductive construction of Chern classes the operation $c_i$
 will be expressed as a sum of $P_i$ 
and some additive operation. 
Part 1 of the Lemma will be enough to prove that $c_i$ is supported
on $K(n)^{i\mod (p^n-1)}$ (which is  Theorem \ref{main}, i)).
Part 2p will be used to show that $p^kP_i$ defines a non-trivial operation modulo $p$.
Also Corollary \ref{cr_kn_add} follows from it, which {\it indirectly} confirms 
and motivates results of Section \ref{op_modp}.
\end{Rk}

\begin{proof}
1. For $j: 1\le j\le p^n-1$ 
denote by $C_{[j]}:=\sum_{r=0}^\infty c_{j+r(p^n-1)}t^{j+r(p^n-1)}$
the specialization of the formal power series $C$ at $c_s=0$, for all $s \neq j \mod (p^n-1)$.

 We need to show that the coefficient of $t^i$ in
 $\log_{K(n)} C_{[j]}=C_{[j]}+\frac{a_1}{p}C_{[j]}^{p^n}+\ldots$ equals to zero.
However, the product of $p^{kn}$ monomials $c_{j+r(p^n-1)}t^{j+r(p^n-1)}$ has the power of $t$ 
equal to $jp^{kn}\equiv j$ modulo $p^n-1$.

2. Let us prove that $p^kP_i$ is an integral polynomial. 
We need to show that the coefficient of $t^i$ 
in $p^kC + a_1p^{k-1}C^{p^n} + \ldots + a_kC^{p^{nk}} + \frac{a_{k+1}}{p}C^{p^{n(k+1)}}+\ldots$
is a polynomial over $\Zp$. 
First $(k+1)$ summands of this expression 
obviously produce an integral polynomial. 
Other summands are of the form $a_{k+r}p^k\frac{C^{p^{n(k+r)}}}{p^{k+r}}$ for $r>0$.
The derivative by $t$ of this expression is equal to $a_{k+r}p^{kn+r(n-1)} C^{p^{n(k+r)}-1} C'$,
which is an integral polynomial with coefficients divisible by $p^{kn+r(n-1)}$.
On the other hand, the coefficient of $t^i$ is multiplied by $i=p^{nk}v$ after differentiation. 
Comparing these two statements and using comparison $\nu_p(i)=nk+\nu_p(v)\le nk+n-1 \le nk+r(n-1)$
 we get that the coefficient of $t^i$ is integral,
and the claim is proved. 

2p. Two cases can be dealt with separately. 

First, assume $\nu_p(v)<n-1$. Then we can prove the proportinality
of polynomials in the statement with 
$P_i$ and $P_v$ instead of $\tilde{P}_i, \tilde{P}_v$.

We have $\nu_p(i)<nk+n-1$ and by the argument from part 2,  
the summands $a_{k+r}p^k\frac{C^{p^{n(k+r)}}}{p^{k+r}}$
do not add anything non-divisible by $p$ for $r>0$.
Thus, the polynomial $-p^kP_i$ modulo $p$ equals to $a_kC^{p^{nk}}[t^i]$. 
The latter is simply $a_kc_v^{p^{nk}}$.

Using the same argument one gets that  $c_v-P_v$ modulo $p$ equals to $C[t^v]$
 and thus is $c_v$. The claim follows.

Second, $\nu_p(v)=n-1$ and assume that $n>1$.
Then again using the same argument we obtain that the polynomial $-p^kP_i$ equals 
$(a_kC^{p^{nk}}+\frac{a_{k+1}}{p}C^{p^{n(k+1)}})[t^i]$ modulo $p$  
and $c_v-P_v$ equals $C[t^v]+\frac{a_1}{p}C^{p^n}[t^v]$ modulo $p$.
Clearly, $(C[t^v])^{p^k}=c_v^{p^{nk}}$ equals $C^{p^{nk}}[t^i]$ modulo $p$.
Our goal is to prove that $\frac{1}{p}C^{p^{n(k+1)}}[t^{p^{nk}v}]$ is equal
to $(\frac{1}{p}C^{p^n}[t^v])^{p^{nk}}$ modulo $p$ when specific variables in $C$ 
are taken to be zero.
As $a_{k+1}\equiv  (a_1)^{k+1} \mod p$, $a_k \equiv (a_1)^k \mod p$
 by Prop. \ref{prop:log_morava}
we will have $p^kP_i = (a_1)^kc_v^{p^{nk}}+(a_1)^{k+1} \frac{1}{p}C^{p^n}[t^v])^{p^{nk}}$
and $(c_v-\tilde{P}_v)^{p^{nk}}\equiv c_v^{p^{nk}} + 
(a_1)^{p^{nk}}(\frac{1}{p}C^{p^n}[t^v])^{p^{nk}}$.
As $(a_1)^{p^{nk}}\equiv a_1 \mod p$ the claimed propotionality would follow.

Let $l$ be s.t. $1\le l \le p^n-1$, $i \equiv l \mod (p^n-1)$ or in other terms, $v= l+s(p^n-1)$. 
From the equality $\nu_p(v)=n-1$ one easily gets that 
there exists $u\in\ZZ$ non-divisible by $p$ s.t. $s=l+up^{n-1}$, i.e. $v=l+(l+up^{n-1})(p^n-1)$.

Define $f(t)$ to be $\sum_{r\ge 0} c_{l+r(p^n-1)}t^r$, thus, $C_{[l]} = t^lf(t^{p^n-1})$.
In this notation we have two equalities:
\begin{equation} 
-p^k\tilde{P}_i = (a_kc_v^{p^{nk}}+\frac{a_{k+1}}{p}t^{lp^{n(k+1)}}f(t^{p^n-1})^{p^{n(k+1)}})[t^{p^{nk}v}] \mod p, 
\qquad
c_v-\tilde{P}_v = c_v+\frac{a_1}{p}t^{lp^n} f(t^{p^n-1})^{p^n}[t^v].
\end{equation}

We need to prove the equality $\frac{1}{p}t^{lp^{n(k+1)}}f(t^{p^n-1})^{p^{n(k+1)}}[t^{p^{nk}v}]
 = (\frac{1}{p}t^{lp^n} f(t^{p^n-1})^{p^n}[t^v])^{p^{nk}} \mod p$.
We may simplify it by changing variables $x=t^{p^n-1}$ and recalling
that $v-lp^n=up^{n-1}(p^n-1)$ and $p^{nk}v-lp^{n(k+1)}=up^{nk+n-1}(p^n-1)$.
Thus, we need to prove that $\frac{1}{p}f(x)^{p^{n(k+1)}}[x^{up^{nk+n-1}}]$
equals to $(\frac{1}{p}f(x)^{p^{n}}[x^{up^{n-1}}])^{p^{nk}}$.
This is already a rather universal equality since it has to be true for any positive integer 
$u$ non-divisible by $p$, any $k>0$
and any series $f$ (as it has coefficients which are independent variables).

To see this recall two simple facts about multinomial coefficients.
\begin{enumerate}
\item For any $m>0$ $p$-valuation $\nu_p\binom{p^m}{r_1,\ldots,r_k}=1$ 
if and only if $r_j=b_jp^{m-1}$ for all $j$, $0<b_j<p$.

In particular, $k$ is at most $p$ and allowing $b_j$ to be equal to 0
we may write these coefficients as $\binom{p^m}{b_1p^{m-1},\ldots,b_pp^{m-1}}$.

\item For any $a_j: 0\le b_j < p$ we have 
$$\frac{1}{p}\binom{p^m}{b_1p^{m-1},\ldots,b_pp^{m-1}} \equiv
\frac{1}{p}\binom{p}{b_1,\ldots,b_p} \mod p.$$
\end{enumerate}

Hence we have the following equalities modulo $p$:

$$\frac{1}{p}f(x)^{p^{n(k+1)}}[x^{up^{nk+n-1}}] = 
\left( \frac{1}{p}f(x)^p[x^u] \right)^{p^{nk+n-1}} =
\left( \frac{1}{p}f(x)^{p^n}[x^{up^{n-1}}] \right)^{p^{nk}}.  $$

The case $n=1$ is different,
as all monomials $\frac{a_{k+r}}{p^r}C^{p^{k+r}}[t^i]$ with $r\ge 0$ 
add something
non-trivial to $-p^kP_i$ modulo $p$, and similarly for $c_v-P_v$ modulo $p$.
However, the same strategy works,
but now we need to show that $(\frac{1}{p^r}C_{[l]}^{p^r}[t^v])^{p^k}$
is equal to $\frac{1}{p^r}C_{[l]}^{p^{k+r}}[t^{p^kv}]$ modulo $p$ for any $r>0$.
Changing variables in the same way as above,
this is equivalent to a universal equality
$(\frac{1}{p^r}f^{p^r}(x)[x^u])^{p^k} \equiv \frac{1}{p^r}f^{p^{k+r}}(x)[x^{p^ku}] \mod p$
which can be proven in a similar fashion.

\end{proof}

Assuming Theorem \ref{main} we deduce the following result,
which will be proven unconditionally later.

\begin{Cr}[cf. Cor. \ref{cr_k0_add}]\label{cr_kn_add}

Let $i=p^{nk}v$ where  $p^n \nmid v$, and 
let $\phi_i, \phi_v$ be generators of integral additive operations 
from $K(n)^*$ to $CH^i\ot\Zp$ and $CH^v\ot\Zp$, respectively.

Then $\phi_i \equiv a (\phi_v)^{p^{nk}} \mod p$ for some $a\in \F{p}^{\times}$.
\end{Cr}
\begin{proof}
If $k=0$ the claim is trivial, so assume $k>0$. 

It is clear that to prove the Corollary we may choose any generators
$\phi_i$, $\phi_v$. In accordance with
Cor. \ref{cr_chern_exist} and using Lemma \ref{lm_add_powers}
we choose them to satisfy the following equalities $\phi_i = p^{k}c_i - p^{k}P_i$,
$\phi_v = c_v - P_v$. 

As additive operations are supported on one component only,
which in this case is $K(n)^v$,
to compare them we may restrict operations to it.
Thus, we get $\phi_v = c_v-\tilde{P}_v$, $\phi_i = p^kc_i-p^k\tilde{P}_i$.

The comparison between $(\phi_v)^{p^{nk}}$ and $\phi_i$ modulo $p$
now follows from part 2p of Lemma \ref{lm_add_powers}.
\end{proof}

\subsection{Construction of Chern classes}\label{chern_constr}

In this section we reduce the existence of operations $c_i:K(n)^*\rarr CH^i\ot\Zp$
satisfying conditions i), ii) of Theorem \ref{main} to Lemma \ref{lift_add},
 which is proved in the next section. 
 
For now let us provide another point of view on 
the Cartan's formula as stated in ii) of Theorem \ref{main}.
We may look at each graded component of $c_{tot}(x+y)$ separately: 
in degree $i$ the derivative of $c_i$ is expressed as a polynomial in external
products of Chern classes of smaller degree. 
For example, for $i=p^n$ we get 
$\partial^1c_{p^n}(x,y):=c_{p^n}(x+y)-c_{p^n}(x)-c_{p^n}(y)=-\frac{a_1}{p}\sum_{j=1}^{p^n-1} \binom{p^n}{i} (c_1(x))^i(c_1(y))^{p^n-i}$.
As the derivative of an operation defines
 it uniquely up to an additive operation, this gives a way of an inductive construction of Chern classes.

It is rather easy to integrate the derivative of the $i$-th Chern class 
as $\QQ$-polynomial $P_i$ in $c_1,\ldots, c_{i-1}$ (notation is consistent with Section \ref{sec_cor}).
For many values of $i$ this polynomial is not integral (Lemma \ref{lm_add_powers}),
and to construct $c_i$ one needs to find an additive operation s.t. the sum of $P_i$ with it
will be integral (cf. Cor. \ref{cr_chern_exist}).
For example, $-\frac{a_1}{p}(c_1)^{p^n}$ has the same derivative
 as the derivative of $c_{p^n}$ predicted
by the Cartan's formula. 
To prove the existence of an integral operation $c_{p^n}$
 we need to find an additive operation $\psi_{p^n}$ from $K(n)^*$ to $CH^{p^n}\ot\QQ$
s.t. $-\frac{a_1}{p}(c_1)^{p^n}+\psi_{p^n}$ is an integral operation.

\begin{proof}[Proof of Theorem \ref{main}]
We construct operations $c_i$, satisfying i) and ii) by induction on $i$
s.t. they satisfy the following property:
\begin{itemize}
\item[{\it iiibis)}] a generator of integral additive operations $\phi_i$ is expressible as an integral polynomial in
$c_1,\ldots, c_i$.
\end{itemize}

Recall that by Prop. \ref{add_rank} for any $j\ge 1$ 
the space $[K(n)^*,CH^j\ot\Zp]^{add}$ of additive operations is a free module over $\Zp$ of rank 1.

{\bf Base of induction.} 
For $1\le i\le p^n-1$ choose $c_i$ to be {\sl  any} generator of $[K(n)^*,CH^j\ot\Zp]^{add}$.

By Prop.\ref{add_grad} the condition i) is satisfied.
As there are no terms in $F_{K(n)}$ of degree bigger than 1 and less than $p^n$
the Cartan's formula can be reformulated as additivity of operations $c_i$ for $1\le i\le p^n-1$. 
Therefore  ii) is satisfied.
The property iiibis) is obviously true.

{\bf Induction step.}
Let $i> p^n-1$. Assume that operations $c_1,c_2,\ldots c_{i-1}$, satisfying  i) and ii), are defined.
In particular, this means that we can calculate derivatives of polynomials in Chern classes $c_j$, $j<i$,
as polynomials in these Chern classes $c_j$, $j<i$.

Define the rational polynomial $P_i\in \QQ[c_1,\ldots,c_{i-1}]$ via the following formula:
$(\log_{K(n)} c_{tot})_i = c_i - P_i$. 
Here $(\log_{K(n)} c_{tot})_i$ is the $i$-th graded component of the series $\log_{K(n)} (c_1+c_2+\ldots)$.

The derivative of $c_i$ as predicted by Cartan's formula
is equal to the derivative of $P_i$ as a polynomial in Chern classes.
Indeed, $(\log_{K(n)} c_{tot})_i$ is predicted to be an additive operation, 
thus $\partial^1((\log_{K(n)} c_{tot})_i)=0=\partial^1(c_i-P_i)$.
This computation is purely algebraic and does not depend on 
any assumptions about operations $c_j$ except for the Cartan's formula.

Therefore if we define $c_i$ as a sum of $P_i$ and some additive operation,
 the condition ii) will be satisfied (up to degree $i$).
 
\begin{Lm}\label{chern_exist}[cf. Cor. \ref{cr_chern_exist}]
Let $\mu_i = \max(0, -\nu_p(P_i))$. 

Then there exist a generator $\phi_i$ of additive operations $[K(n)^*, CH^i\ot\Zp]^{add}$,
s.t. operation $c_i:K(n)^*\rarr CH^i\ot\QQ$ defined by the formula
$c_i = P_i(c_1,\cdots,c_{i-1})+\frac{\phi_i}{p^{\mu_i}}$
acts integrally on products of projective spaces.
\end{Lm}

The existence of operation $c_i:K(n)^*\rarr CH^i\ot\Zp$
satisfying conditions i), ii), iiibis) follows from this Lemma.
Indeed, the COT yields existence of integral operation $c_i$,
s.t. $c_i\ot id_\QQ$ (which is loosely denoted by $c_i$ in the Lemma)
 satisfies the formula in Lemma \ref{chern_exist}.
The Cartan's formula will be true for integral $c_i$, 
since it is an equality between two operations
which can be checked on products of projective spaces. 
As there is no torsion in values of our theories on products of projective spaces,
the statement can be checked rationally where it is true by the formula defining $c_i$.

Let us show now that the condition i) is satisfied.
Choose $j\neq i \mod{p^n-1}$.
By Lemma \ref{lm_add_powers} the specialization $\tilde{P}_i$ of the polynomial $P_i$ 
at $c_s=0$, for $s\neq j \mod(p^n-1)$ is equal to zero.
By the induction assumption this is what happens to the classes
$c_s$ as above for $s<i$ when restricted to $K(n)^j$.
By Prop. \ref{add_grad} the additive operation $\frac{\phi_i}{p^{\mu_i}}$ is supported on $K(n)^i$.
Therefore by the formula defining $c_i$ this operation is zero when restricted to $K(n)^j$
by the COT.

The condition iiibis) is satisfied by the choice of $\mu_i$ as explained in Section \ref{sec_cor}.

We finish this section by reducing Lemma \ref{chern_exist} to a result 
on additive operations $[K(n)^*,CH^*/p]^{add}$, Lemma \ref{lift_add},
which is proved in the next section. In fact, much stronger version of it will be proved
saying that the reduction modulo $p$ of {\it many} integral operations, whenever it is additive,
 is proportional to the reduction of an integral additive operation.

\begin{Lm}\label{lift_add}
Suppose that for some $\mu >0$, $a\in\QQ$ operation $p^{\mu}P_i+a\phi_i$,
where $\phi_i$ is a generator of integral additive operations $[K(n),CH^i\ot\Zp]^{add}$,
acts integrally on products of projective spaces and, thus, (by the COT)
defines an integral operation $\pi$.

Then $\pi$ is proportional to $\phi_i$ modulo $p$.
In particular, operation $\pi$ is additive modulo $p$.
\end{Lm}

As a matter of fact let us show first that if $\pi$ is integral, then it is additive modulo $p$.
It is enough to show that its derivative is zero modulo $p$ as an integral polynomial in Chern classes.
Recall that the derivative of $p^\mu P_i$ equals to $p^\mu \partial^1c_i$ by construction.
Here $\partial^1c_i$ is a formal notation meaning an integral polynomial in Chern classes
$c_1,\ldots, c_{i-1}$ which is predicted by the Cartan's formula. 
Poly-operation $p^\mu \partial^1c_i$ is equal to zero modulo $p$, since $\mu>0$.
The derivative of $a\phi_i$ is zero as well.

\begin{proof}[Lemma \ref{chern_exist} follows from Lemma \ref{lift_add}.]

Fix any generator $\phi_i$ of integral additive operations $[K(n)^*,CH^i\ot\Zp]^{add}$.
Let us prove the following statement by (finite) induction on $r$ using Lemma \ref{lift_add}.

{\bf Claim.}
For $0\le r \le \mu_i$ there exist $\alpha_r\in \Zp^{\times}$
 s.t. $p^{\mu_i-r}P_i +\frac{\alpha_r}{p^r}\phi_i$ is an integral operation.

{\bf Base of induction ($r=0$).} 
Recall that $\mu_i$ is chosen so that $p^{\mu_i}P_i$ is integral,
and so if $r=0$, one may take $\alpha_0=1$.

{\bf Induction step.}
Suppose we have shown that $p^{\mu_i-r}P_i +\frac{\alpha_r}{p^r}\phi_i$ is an integral operation.

It follows from a discussion above that if $r<\mu_i$, 
then Lemma \ref{lift_add} is applicable to the operation $\pi$
defined by $p^{\mu_i-r}P_i +\frac{\alpha_r}{p^r}\phi_i$ modulo $p$.
Therefore the operation $p^{\mu_i-r}P_i +\frac{\alpha_r}{p^r}\phi_i$ 
equals to $b \phi_i$ modulo $p$ for some $b\in\Zp$.
The operation $p^{\mu_i-r}P_i+(\frac{\alpha_r}{p^r}-b)\phi_i$ is zero modulo $p$ and hence 
the operation $p^{\mu_i-(r+1)}P_i+\frac{\alpha_r-bp^r}{p^{r+1}}\phi_i$ is an integral operation.
 Define $\alpha_{r+1} = \alpha_r-bp^r$.

If $r>0$ then $\alpha_r-bp^r \in \Zp^{\times}$, since $\alpha_r\in \Zp^\times$, $b\in \Zp$,
 and the induction step is proved.

However, for the induction step with $r=0$ we need to show that $b\neq \alpha_0 \mod p$.

{\bf Additional details on the induction step $r=0 \rarr r=1$.} 

Note that this applies only if $p^n|i$, as otherwise by Lemma \ref{lm_add_powers} $\mu_i=0$
and the induction stops at the base.

As we want $\alpha_1$ not to be $p$-divisible, 
which is the same as $b \neq \alpha_0 \mod p$,
we need to show that 
the operation $p^{\mu_i}P_i$ should not be equal to zero modulo $p$.
To prove it use Lemma \ref{lm_add_powers}, 2p which says that 
$-p^{\mu_i}P_i$ is proportional to $(c_v-P_v)^{p^{nk}}$ modulo $p$. 
However, by induction assumption of the construction of Chern classes
 property iiibis) is satisfied for them, i.e.
$c_v-P_v$ is a generator of integral additive operations. 
Thus, it is not equal to zero modulo $p$,
as otherwise using the COT one could divide it by $p$ and yield another integral additive operation.
Powers of $\phi_v \mod p$ are not zero as well by the COT,
since the coefficient ring of the target theory has no zero divisors.
The induction step and the Claim are proven.

Now we can define $c_i :=  P_i +\frac{\alpha_{\mu_i}}{p^{\mu_i}}\phi_i$,
 which is proved to be an integral operation. 
 As $\alpha_{\mu_i} \in \Zp^{\times}$ the operation $\alpha_{\mu_i}\phi_i$
  is a generator of integral additive operations
 and Lemma \ref{chern_exist} is proved to follow from Lemma \ref{lift_add}.
\end{proof}

\renewcommand{\qedsymbol}{}
\end{proof}

\subsection{Additive operations to Chow groups modulo $p$.}\label{op_modp}

In this section we prove Lemma \ref{lift_add} as Corollary \ref{grad_modp} of the study of additive
operations $[K(n), CH^*/p]^{add}$. 
Though the proof of Prop. \ref{grad_coef} from which the needed Corollary follows
is independent of the induction construction of the previous section,
to avoid misunderstanding we try to escape using $i$ as a varible
where the construction of Chern classes is not involved.

The proof is based on a general discussion about the system of linear equations,
which defines additive operations according to the Classification of Operations Theorem.
Roughly speaking it goes as follows. This system is finite-dimensional
(for operations to a particular component of $CH^*/p$) and is upper-triangular when written 
in a naturally chosen basis.
Over the rings $\Zp$ and $\QQ$ the diagonal coefficients
 of this system are non-zero,
and therefore the space of solutions is 1-dimensional (cf. Prop. \ref{add_rank}).
However, over $\F{p}$ many of the diagonal coefficients 
are zero which leads to a higher dimension of the space of solutions
and to the existence of additive operations to $CH^*/p$
which are not liftable as additive operations to $CH^*\ot\Zp$.
It turns out that
 for a rather natural set of additive operations ({\sl gradable} operations,
  Definition \ref{def_grad})
 it is possible to use equations with zeros on the diagonal 
 to express all variables in terms of one of them.
 In other words, 
 one can transform this system into an upper-triangular one without zeros on the diagonal,
 but the choice of the new basis for this transformation is not so natural 
 in the coordinates we work with.
 Anyway this proves that the space of gradable additive operations to $CH^i/p$ is 1-dimensional 
 (Cor.\ref{grad_modp}) and it is easy to show 
 that it is generated by a reduction of an integral additive operation.
Lemma \ref{lift_add} then easily follows.

{\it Notation (cf. Section \ref{sec:CAOT}).} 
Let $A$ denote one of the rings $\Zp, \F{p}, \QQ$.
 If $\phi:\tilde{K}(n)^*\rarr CH^*\ot A$ is an operation, 
denote by $G_l=G_l(t_1,\ldots, t_l)\in A[t_1,\ldots, t_l]$ the value of $\phi$ on
$\prod_{i=1}^l c_1^{\tilde{K}(n)}(\mathcal{O}(1)_i)
:=z_1\cdots z_l\in \tilde{K}(n)(\prod (\mathbb{P}^\infty)^{\times l})$
expressed as a symmetric polynomial in $t_j=c_1^{CH}(\mathcal{O}(1)_j)$.
Here $\mathcal{O}(1)_j$ is the pullback of the canonical line bundle on the $j$-th component 
of the product of projective spaces.

\begin{Rk}\label{Gdiv}
From continuity of operations it follows that $G_l$ is divisible by $t_1\cdots t_l$.
\end{Rk}

\begin{Def}\label{def_grad}
An operation $\phi:\tilde{K}(n)^*\rarr CH^i\ot A$ is called gradable,
if for any $l\ge 1$ the symmetric polynomial $G_l(t_1,\ldots, t_l)$ admits only monomials 
where every variable has its power equal to 1 modulo $p^n-1$.
\end{Def}

The following is straight-forward.

\begin{Prop}\label{grad}

\begin{enumerate}
\item The sum of gradable operations is gradable;
\item product of $N$ gradable operations is gradable if $N\equiv 1 \mod {p^n-1}$;
\item an operation $\phi:K(n)^*\rarr CH^i\ot\Zp$ is gradable iff $\phi\ot id_\QQ: K(n)\rarr CH^i\ot\QQ$ is gradable.
\end{enumerate}
\end{Prop}

\begin{Prop}\label{prop:add_grad}
All additive operations from $K(n)^*$ to $CH^*\ot\Zp$ are gradable.
\end{Prop}
\begin{proof}
Any additive operation from $K(n)^*$ to $CH^j\ot\Zp$ 
 is rationally propotional to the component of the Chern character $ch_j$.
Due to Prop. \ref{grad}, (3), it is enough to prove that all components of the Chern character
are gradable.

Since the Chern character is multiplicative, we have $G_l(t_1,t_2,\ldots,t_l) =\prod_j G_1(t_j)$.
It is neccessary and sufficient for $ch$ to be gradable that 
 $\gamma(t):=G_1(t)$ admits only monomials in $t$ of the power equal to 1 modulo $p^n-1$.

The series $\gamma$ defines a morphism from the FGL $F_{K(n)}$ to the additive FGL.
Thus, by definition of the logarithm, $\gamma$ is the composition inverse of $\log_{K(n)}$.
One may consider the inverse to the homogenous series $x+v_n\frac{a_1x^{p^n}}{p}+\ldots$ 
over $\Zp[v_n]$, where $\deg v_n = 1-p^n$ and $\deg x =1$. 
Its inverse is homogenous as well. However, as $\log_{K(n)}$ can be obtained
from this series by setting $v_n=1$, $\gamma$ can be obtained from its inverse 
by the same procedure. Therefore, it is 'gradable', 
the operation $ch$ is gradable and the proposition is proved.
\end{proof}

\begin{Rk}
One can show that not all operations $K(n)^*\rarr CH^*\ot\Zp$ are gradable, e.g. $\phi_1^2$ is not. 

For $n>1$ not all additive operations $K(n)^*\rarr CH^*/p$ are gradable either,
 the easiest example being $\phi_1^p$.
\end{Rk}

\begin{Prop}\label{chern_grad}
Let $i\ge p^n$.
Assume that operations $c_1,\ldots, c_{i-1}:K(n)^*\rarr CH^*\ot\Zp$,
 satisfying i), ii) and iiibis) of Theorem \ref{main} exist.

Then these operations are gradable.

Denote by $P_j=-(\log_{K(n)} c_{tot})_j+c_j$, $j< i$, 
a rational polynomial in Chern classes $c_1,\ldots, c_{j-1}$.

Then $aP_j+b\phi_j$ defines a gradable operation for any $a,b\in \QQ$.
\end{Prop}
\begin{proof}
{\bf Base of induction.} 
Operations $c_1, c_2, \ldots, c_{p^n-1}$ are additive 
and are proved to be gradable in Prop. \ref{prop:add_grad}.

{\bf Induction step.}
The operation defined by the polynomial $P_j$ is gradable
as follows from the gradability of $c_1, \ldots, c_{j-1}$ and Prop. \ref{grad}, (2).
Indeed, any monomial of $\log_{K(n)} c_{tot}$ is a product of $p^{kn}$ Chern classes for $k\ge 0$
and $p^{kn}\equiv 1 \mod {p^n-1}$.

Recall from Section \ref{chern_constr} that rationally $c_j$ is equal to a sum of a multiple of $P_j$ 
and a rational additive operation $\psi_j$. 
The operation $P_j$ is gradable by induction and $\psi_j$ is gradable as an additive operation.
By Prop. \ref{grad} the induction step is proved.

Operation $aP_j+b\phi_j$ is a sum of gradable operations and therefore gradable.
\end{proof}

Now we rewrite the CAOT for the case $A^*=K(n)^*$, $B^*=CH^m/p$.

\begin{Th}
Additive operations $[\tilde{K}(n)^*, CH^m/p]^{add}$ are in 1-to-1 correspondence with 
the set of symmetric polynomials $G_l\in \F{p}[t_1,\ldots, t_l]\cdot (t_1\cdots t_l)$ 
(cf. Remark \ref{Gdiv})
 of degree $m$ for $l\ge 1$ which satisfy the following equations:

$$ G_l (t_1, t_2, \ldots, t_{l-1}, u+v) 
- \sum_{j,k\ge 0} a_{jk} G_{j+k+l-1}(t_1, \ldots, t_{l-1},u^{\times j}, v^{\times k})  = 0, \eqno (A_l)$$
where $a_{jk}$ are the coefficients of the formal group law of $K(n)$:
$F_{K(n)}(x,y)=\sum_{j,k} a_{jk} x^jy^k$.
\end{Th}
\begin{Rk}
If we investigated operations from $K(n)^*$, and not from $\tilde{K}(n)^*$, 
then there would be also a polynomial $G_0$ defining an operation. It is equal to zero in our case.
\end{Rk}

Let $\vec{r}=(r_1,\ldots, r_k)$ be a partition of $m$.
We do not imply in our notations that numbers $r_j$ are ordered. 
So, $(1,2,3)$ and $(2,1,3)$ are both acceptable notations of the same partition.
Denote by $\alpha^{(s)}_{(r_1,\ldots,r_k)}$ or $\alpha^{(k)}_{\vec{r}}$
 a coefficient of any monomial $t_{j_1}^{r_1}\cdots t_{j_k}^{r_k}$ 
in the symmetric polynomial $G_k$ where $\{j_1,\ldots,j_k\}=\{1,2,\ldots,k\}$.
 We will call these coefficients {\sl variables} as we consider
them unknown in equations $A_l$.
The system of equations $(A_l)$ is a finite-dimensional 
linear homogenous system on $\alpha^{(k)}_{(r_1,\ldots,r_k)}$.

\begin{Prop}\label{upper_triangular}
Let the variables $\{\alpha^{(k)}_{\vec{s}}\}$ 
parameterized by all $\vec{s}$ -- $k$-partitions of $m$
  satisfy the system $(A_l)$.
  
 Let $a,b, r_1,\ldots, r_{l-1}>0$ be such that
  $\vec{r}:=(r_1,\ldots,r_{l-1},a+b)$ is a partition of $m$.

Then the coefficient of the monomial $t_1^{r_1}\cdots t_{l-1}^{r_{l-1}} u^a v^b$ in the equation $A_l$ 
lies in $\F{p}\alpha^{(l)}_{\vec{r}}+\oplus_{k>l} \F{p}\alpha^{(k)}_{(s_1,\ldots,s_k)}$.

Moreover, this coefficient lies in $\oplus_{k>l} \F{p}\alpha^{(k)}_{(s_1,\ldots,s_k)}$
if and only if the partition $\vec{r}$ is $p$-special, i.e. $r_q = p^{w_q}$.

In particular, if $\vec{r}$ is not $p$-special,
then $\alpha^{(l)}_{\vec{r}}$ is expressible in terms of $\alpha^{(k)}_{\vec{s}}$ for $k>l$. 
\end{Prop}

\begin{proof}
The variables $\alpha^{(k)}_{\vec{s}}$ with $k<l$ do not appear in the equation $A_l$.
Thus, it is enough to show that the variables $\alpha^{(l)}_{\vec{s}}$ do not appear 
as coefficients of the monomial $t_1^{r_1}\cdots t_{l-1}^{r_{l-1}}u^a v^b$ 
in the equation $A_l$ for $\vec{s}\neq \vec{r}$.

In fact these variables can appear only in the expression 
$G_l (t_1, t_2, \ldots, t_{l-1}, u+v)-G_l(t_1, t_2, \ldots, t_{l-1}, u)-G_l (t_1, t_2, \ldots, t_{l-1}, v)$
 of the equation $A_l$. That is we have to look at 
$\alpha^{(l)}_{(s_1,\ldots,s_l)} t_1^{s_1}\cdots t_{l-1}^{s_{l-1}} \left((u+v)^{s_l}-u^{s_l}-v^{s_l} \right)$, 
which obviously does not contain the monomial under investigation if
partitions $\vec{s}$ and $\vec{r}$ are different.

The variable $\alpha^{(l)}_{\vec{r}}$ appears with a non-zero coefficient
in the expression
 $$G_l (t_1, t_2, \ldots, t_{l-1}, u+v)-G_l(t_1, t_2, \ldots, t_{l-1}, u)-G_l (t_1, t_2, \ldots, t_{l-1}, v)$$
 whenever the corresponding monomials $t_{j_1}^{r_1}\cdots t_{j_l}^{r_l}$ 
 are not additive over $\F{p}$.  
 This condition exactly means that the partition $\vec{r}$ is not $p$-special.
\end{proof}

By the definition of a gradable operation variables $\alpha^{(k)}_{\vec{r}}$ corresponding to it
have to be zero whevener there exist $r_q\in \vec{r}$ s.t. $r_q \neq 1 \mod (p^n-1)$.
We now investigate when the variables $\alpha^{(l)}_{\vec{r}}$ corresponding
to $p$-special partitions can be non-zero for gradable operations.
Note, that these are precisely the variables
which are in charge of the possible non-liftable additive operations.

\begin{Prop}\label{deg_coef}
The number $p^r$ is equal to 1 modulo $p^n-1$ if and only if $r$ is divisible by $n$.

Therefore variables $\alpha^{(l)}_{\vec{r}}$ corresponding 
to $p$-special partition $\vec{r}$ are zero for gradable operation
whenever
there exists $r_q=p^{w_q}$ s.t. $w_q$ is not divisible by $n$.
\end{Prop}
\begin{proof}
From the equality $r=1+v(p^n-1)=p^m$ we get that $p^n-1|p^m-1$ and hence $n|m$
($p^m \equiv p^{m\mod n} \mod (p^n-1)$ and thus $m \equiv 0 \mod n$).
By definition of gradable operations the claim follows.
\end{proof}

We will call partitions $(p^{ns_1}, \ldots, p^{ns_l})$ -- $p^n$-special.

\begin{Prop}\label{grad_coef}
Let $\{\alpha^{(k)}_{(r_1,\ldots,r_k)}\}$ satisfy
the system of equations $(A_l)$ and let $\alpha^{(j)}_{\vec{r}}=0$
whenever there exists $r_q \neq 1 \mod (p^n-1)$, $1\le q\le l$.
In other words, $\{\alpha^{(k)}_{\vec{r}}\}$ correspond to a gradable additive operation.

Let  $\vec{r}:=(r_1,\ldots, r_l)$ be a $p^n$-special partition
s.t. there are at least $p^n$ equal numbers among $r_q$, $1\le q \le l$.

Then using equations $A_{l-p^n+1}$ and $A_l$ one may express
 $\alpha^{(l)}_{\vec{r}}$ in terms of $\alpha^{(k)}_{\vec{s}}$ for $k>l$.
\end{Prop}
\begin{proof}
As there are at least $p^n$ equal numbers in the partition $\vec{r}$,
 it follows that $l> p^n-1$.
Denote by $r_{l-p^n+1}=r_{l-p^n+2}=\ldots = r_l = p^{sn}$ for some $s\ge 0$.

Recall that modulo $p$ the formal group law of the $n$-th Morava K-theory looks as
$$x+y-\frac{a_1}{p}\sum_{j=1}^{p-1} \binom{p}{j}x^{p^{n-1}j}y^{p^{n-1}(p-j)}
+\textrm{higher degree terms},$$
i.e. $a_{t,p^n-t}\neq 0 \mod p$ if and only if $t=jp^{n-1}$, $0<j<p$.

Look at the coefficient of the monomial 
$t_1^{r_1}\cdots t_{l-p^n}^{r_{l-p^n}}u^{p^{sn+n-1}}v^{p^{sn+n-1}(p-1)}$ in the equation $A_{l-p^n+1}$.
The variable $\alpha^{(l)}_{(r_1,\ldots,r_l)}$ appears non-trivially as a coefficient of this monomial
 coming from the term 
 $$a_{p^{n-1},(p-1)p^{n-1}}G_{l}(t_1,\ldots, t_{l-p^n},u^{\times p^{n-1}}, v^{\times (p-1)p^{n-1}}).$$ 
 
Here we have used the assumption that there are at least $p^n$ equal numbers in $(r_1,\ldots,r_l)$.
Otherwise we would get some multiple of $\alpha^{(l)}_{(r_1,\ldots,r_l)}$,
because $G_l(t_1,\ldots, t_l)$ is symmetric and monomials may 'glue' in 
$G_{l}(t_1,\ldots, t_{l-p^n},u^{\times p^{n-1}}, v^{\times (p-1)p^{n-1}})$.

The first claim is that variables $\alpha^{(k)}_{\vec{m}}$ with $k<l$ 
do not appear in the coefficient of this monomial. 
In equation $A_{l-p^n+1}$ these variables may appear only for $k=l-p^n+1$,
however they do not as monomials in question are additive (Prop. \ref{upper_triangular}).

The second claim is that variables $\alpha^{(l)}_{(m_1,\ldots,m_l)}$
 which appear in the coefficient of the monomial in question
are only those which come from non-additive monomials, i.e. $p$ does not divide $m_i$ for some $1\le i\le l$.
These can be expressed in terms of variables $\alpha^{(k)}_{\vec{s}}$ with $k>l$ 
via equation $A_l$ by Prop. \ref{deg_coef} and so the proof will be finished.

Indeed, the only variables coming from additive gradable monomials are 
$\alpha^{(l)}_{(p^{nk_1},\ldots,p^{nk_l})}$. 
In the equation $A_{l-p^n+1}$ these may come from the term
$a_{wp^{n-1},(p-w)p^{n-1}}G_{l}(t_1,\ldots, t_{l-p^n},u^{\times wp^{n-1}}, v^{\times (p-w)p^{n-1}})$ 
for some $w: 1\le w \le p-1$. 
To appear as a coefficient of the monomial 
$t_1^{r_1}\cdots t_{l-p^n}^{r_{l-p^n}}u^{p^{sn+n-1}}v^{p^{sn+n-1}(p-1)}$
 we need $p^{nk_q} = r_q$ for $q<l-p^n$, 
and moreover $\sum_{q=l-p^n+1}^{l} p^{nk_q} = p^n p^{ns}=p^{n(s+1)}$.

It is easy to see that the latter equation
has a unique solution with $k_{l-p^n+1}=\ldots=k_{l}=s$.
This corresponds to the coefficient we were interested in,
 so no other variable $\alpha^{(l)}_{(p^{nk_1},\ldots,p^{nk_l})}$ appears in the coefficient 
 of the monomial in question.
\end{proof}

\begin{Cr}\label{grad_modp}
The $\F{p}$-vector space of gradable additive operations is 1-dimensional
and is generated by $\phi_m \mod p$,
where $\phi_m:K(n)^*\rarr CH^m\ot\Zp$ is a generator of integral additive operations. 

In particular, Lemma \ref{lift_add} and Cor. \ref{cr_kn_add} are true.
\end{Cr}
\begin{proof}
To prove this corollary we 'solve' the system of equations $(A_l)$ for coefficients of gradable operations. 
Namely, we express all variables $\alpha^{(l)}_{(r_1,\ldots,r_l)}$ as multiples of one of
them. This will yield that the dimension of the space of solutions
is bounded above by 1. 

By Prop. \ref{prop:add_grad} the operation $\phi_m \mod p$ is gradable.
If it were zero, then it would follow from the COT
that $\frac{1}{p}\phi_m$ is an integral additive operation
which contradicts the assumption that $\phi_m$ is a generator.
This yields that the dimension of the space of gradable additive operations 
is bounded below by 1.

We say that the variable $\alpha$ is expressible in terms of some variables $\beta_s$
if there exist a consequence of the system of equations $(A_l)$ which
looks like this: $\alpha = \sum b_s \beta_s$, where $b_s\in\F{p}$.
From Prop. \ref{upper_triangular} and  \ref{grad_coef} 
we know that the variable $\alpha^{(l)}_{(r_1,\ldots,r_l)}$
is expressible in terms of variables with a bigger superscript 
whenever there exists $r_q \neq p^{ns_q}$ for any $s_q$ or
 if there are at least $p^n$ equal numbers among $(r_1,\ldots,r_l)$.

Consider now other variables $\alpha^{(l)}_{(r_1,\ldots,r_l)}$ corresponding to $p^n$-special partitions 
which do not satisfy this condition.
Denote by $m_k$ the number of $p^{kn}$ among $(r_1,\ldots,r_l)$, where $k\ge 0$.
We know that $0\le m_k<p^n$ for any $k$,
 and we know that $m=m_0 + m_1 p^n + m_2 p^{2n} + \ldots + m_s p^{sn}$
 because it is a partition of $m$.
Thus $m_i$ are uniquely defined by $m$ as they are the digits of $m$ in the $p^n$-ary digit system.
This means that the 'exceptional' variable 
which can not be expressed in terms of variables 
with higher superscript by propositions \ref{upper_triangular}, \ref{grad_coef}
is unique. 
(Note that we do not claim that it could not be expressed in this way by some other argument.)

If $m<p^n$, then the space of gradable additive operations
is 1-dimensional as it has to be determined by the only coefficient $\alpha^{(m)}_{(1,1,\ldots,1)}$
which is the 'exceptional' one.
If $m\ge p^n$, then the variable $\alpha^{(m)}_{(1,1,\ldots,1)}$ 
equals 0 due to the Prop. \ref{grad_coef}.
Any variable except for the exceptional one can be expressed 
in terms of the exceptional one and $\alpha^{(m)}_{(1,1,\ldots,1)}$.
The dimension of the space of gradable additive operations is thus proved to be 1.

To prove Lemma \ref{lift_add} note that due to Prop. \ref{chern_grad} expression $p^{\mu}P_i+a\phi_i$ 
defines a gradable operation. It was shown in Section \ref{chern_constr} that if it is integral
then it is additive modulo $p$. Therefore it is propotional to $\phi_i$ modulo $p$.

Cor. \ref{cr_kn_add} follows as $(\phi_v)^{p^{kn}}$ is a gradable additive operation
 by Prop. \ref{grad} and is non-zero as follows from the COT since
the coefficient ring of $CH^*/p$ has no zero divisors.
Thus, operations $(\phi_v)^{p^{kn}}$ and $\phi_{vp^{kn}}$ are proportional.
\end{proof}

All the steps of the induction construction of Chern classes $c_i$
are now verified, and so, parts i), ii) and iiibis) of Theorem \ref{main} are proven.

One can show after the proof of iii) of Theorem \ref{main} 
(and in the same manner as the proof in Section \ref{chern-base})
 that the only liftable additive operations $K(n)\rarr CH^*/p$ are gradable operations 
and their $p$-primary  powers (however $p$-primary powers do not lift in general as additive operations).
From this it is easy to show that there are many additive operations to $CH^*/p$ which are not liftable.
Let us produce now a construction of such operations 
with a particular example which is proved to be non-liftable directly.

It is well-known (cf. \cite[Th. 6.6]{PS_Vish1}) that
 to any $p$-partition $(p^{s_1},\ldots, p^{s_i})$ of $j$ corresponds an additive operation
$CH^i/p\rarr CH^j/p$
which sends the product of $i$ first Chern classes $t_1\cdots t_i$
to the symmetrization of the monomial $t_1^{p^{s_1}}\cdots t_i^{p^{s_i}}$
(as in the notations of Section \ref{sec:CAOT}).

It is clear that the composition of such an operation with a gradable operation is
not always gradable. 
We expect however that all additive operations from $K(n)^*$ to $CH^*/p$ 
are generated by the action of the Steenrod algebra 
on gradable operations though we do not prove it here. 

For example, let $n>1$, and consider an additive operation 
$Q:CH^2/p\rarr CH^{p+1}/p$ which sends  $t_1t_2$ to $t_1^pt_2 + t_1t_2^p$.
The composition $Q\circ c_2$ is a non-zero additive operation, which is supported on $K(n)^2$. 
Assume that $p\neq 2$.  From iii) of Theorem \ref{main} it follows
that there is the only (up to a scalar) integral operation from $K(n)^2$ to $CH^{p+1}\ot\Zp$ 
which is $c_2^{\frac{p+1}{2}}$. One easily checks that this is not additive modulo $p$ therefore not proportional to $Q\circ c_2$, and therefore $Q\circ c_2$ is not liftable.

\subsection{Chern classes freely generate all operations}\label{chern-base}

In this section we prove the statement iii) of Theorem \ref{main}.

In Sections \ref{chern_constr}, \ref{op_modp} we have constructed
 operations $c_i:K(n)^*\rarr CH^i\ot\Zp$, $i\ge 1$,
 satisfying properties i), ii) of Theorem \ref{main} and iiibis) of Section \ref{chern_constr}.
The latter says that generators of additive integral operations $\phi_i$ 
can be expressed as an integral polynomial in $c_1,\ldots, c_i$.
We now prove that Chern classes $c_i$ generate freely the ring of integral operations to Chow groups. 
In fact, we get a more general statement that external products of Chern classes
freely generate rings of poly-operations (Prop. \ref{prop_chern_base}).

The proof is independent of the construction of Chern classes
 and uses only the notion of a gradable operation (Def. \ref{grad}) 
 and the fact that integral additive operations are gradable (Prop. \ref{prop:add_grad}).

The starting point for the proof of iii) is the natural inclusion
$[\tilde{K}(n)^*, CH^*\ot\Zp]\subset [\tilde{K}(n)^*, CH^*\ot\QQ]$ provided by the COT. 
The latter space is freely generated by components of the Chern character (Prop. \ref{ch-base}),
and it is easy to show that it is freely generated by Chern classes 
(Prop. \ref{chern_polybase} below).
Thus, if we have an integral operation it can be uniquely expressed as a rational series 
in Chern classes. 
The problem is then to prove that this is in fact an integral series. 
To do this we study derivatives of this operation 
and, roughly speaking, reduce everything to the case of additive poly-operations (Lemma \ref{polyderint}).

\begin{Prop}\label{chern_polybase}
Chern classes are free generators of the $\QQ$-algebra
 of poly-operations from $\tilde{K}(n)^*$ to $CH^*\ot\QQ$.
\end{Prop}
\begin{proof}
From Prop. \ref{ch-base} one easily sees that it is enough to show
 that components of the Chern character can be expressed as polynomials in Chern classes,
 and Chern classes can be expressed as polynomials in components of the Chern character.
 
Indeed, by iiibis) a generator of integral additive operations $\phi_i$ is
equal to an integral polynomial in Chern classes. 
Rationally the operation $\phi_i$ is a multiple
 of $ch_i$ and 
 thus $ch_i$ is a rational polynomial in  $c_1,\ldots, c_i$.

 The operations $c_i$ can be expressed as polynomials in $ch_j$, $j\le i$,
  by the construction.
\end{proof}

Recall from Section \ref{prelim} that we denote by $\odot$ the external product of operations.
Proposition \ref{chern_polybase} thus can be written as
$[\tilde{K}(n)^{*\times r}, CH^*\ot\QQ]=\QQ[[c_1\ldots, c_n, \ldots]]^{\odot r}$. 

We call $\psi_i$ the $i$-th component of poly-operation
 $\psi_1\odot \psi_2 \odot \cdots \odot \psi_r$.

\begin{Prop}\label{polyaddbase}
Let $A^*$ be a theory of rational type, let $B^*$ be a g.o.c.t. 
and let $\{\phi_s\}_{s\in S}$ be a set of $B$-linearly independent additive operations 
from $A^*$ to $B^*$.

Then external products of operations 
$\{\phi_{s_1}\odot\phi_{s_2}\odot \cdots\odot \phi_{s_r}\}_{s_1,s_2, \ldots, s_r \in S}$
are $B$-linearly independent.
\end{Prop}
\begin{proof}
We prove the statement by induction on arity of poly-operations. 

{\bf Base of induction} is the assumption of the Proposition. 

{\bf Induction step}. Suppose we know the statement of the Proposition
 for $i$-ary poly-operations, where $i<r$.
Consider a non-trivial linear combination 
$T:=\sum \beta_{(s_1,s_2,\ldots,s_r)} \phi_{s_1}\odot \phi_{s_2} \odot \cdots \odot \phi_{s_r}$.
Choose a monomial with a non-zero coefficient in this linear combination 
$\beta \phi_{s_1}\odot \phi_{s_2} \odot \cdots \odot \phi_{s_r}$,
 and collect all the terms in $T$ which differ from it only in the first component.
Denote it by $ R\odot \phi_{s_2} \odot \cdots \odot \phi_{s_r} 
:= (\sum_{l} \beta_{(l,s_2,\ldots,s_r)}\phi_{l})\odot \phi_{s_2} \odot \cdots \odot \phi_{s_r}$.

By assumption of the Proposition using the COT
we obtain that there exists $x$ in $A^*((\mathbb{P}^\infty)^{\times k})$ for some $k\ge 0$,
s.t. $R(x)\neq 0$. 
Restrict the poly-operation $T$ in the first component to this element.
Thus, we get a natural transformation $T_x:=T(x,-)$ from the functor
$\times_{i=1}^{k-1} A^*$ to the functor $B^*\circ ((\mathbb{P}^\infty)^{\times k}\times \prod)$.

Note that $B^*(X\times (\mathbb{P}^\infty)^{\times k})\cong B^*(X)\ot_B B[[z_1,\ldots, z_k]]$
for any $X$ by the projective bundle theorem.
Choose any $B$-linear projection $p:B^*(\mathbb{P}^\infty)^{\times k}\cong B[[z_1,\ldots,z_k]]\rarr B$
s.t. $p(R(x))\neq 0$ and compose it with $T_x$ to get an $(k-1)$-ary poly-operation
$(id\otimes p)\circ T_x$.
It can be expressed as a sum of $p(R(x)) \phi_{s_2} \odot \cdots \odot \phi_{s_r}$ and 
a linear combination of external products of $\phi_i$ which does not contain this summand.
By induction assumption this is a non-trivial poly-operation, and therefore $T$ is non-trivial as well.
\end{proof}
\begin{Cr}\label{cr_polyadd}
External products of additive operations $\{\phi_i^{p^s}\}_{\{p^n \nmid i\}}$ 
are linearly independent over $\F{p}$ as poly-operations to $CH^*/p$.
\end{Cr}

Note that by Cor. \ref{cr_kn_add} (unconditionally proved in Cor. \ref{grad_modp})
the operation $\phi_j^{p^n}$ is proportional to $\phi_{jp^n}$ modulo $p$ for any $j\ge 1$,
which explains a somewhat strange set of additive operations in the statement.

\begin{proof}
By Prop. \ref{polyaddbase} it is enough to show that these additive operations
are linearly independent.

Suppose we have a linear combination $\sum_i \alpha_i\phi_{i}^{p^{r_i}}$,
which is zero modulo $p$ as an operation. 
Let $r=\min_i r_i$. If $r>0$ then consider $p^r$-th root of this expression.
If it were not a trivial operation, then from the COT
it would follow that its $p$-primary powers are non-trivial 
(as there are no $p$-nilpotents in the coefficient ring of a target theory). 
Thus, we obtain a relation in which there is a unique summand $\alpha_i \phi_i$, $\alpha_i\neq 0$,
which is a gradable operation (see Def. \ref{def_grad})  by Prop. \ref{prop:add_grad}.

Since $i$ is not divisible by $p^n$, then by degree reasons all other summands look
like $\alpha_j \phi_j^{p^{r_j}}$ where $jp^{r_j}=i$, and therefore $n\nmid r_j$. 
It follows from the fact that $\phi_j$ are gradable
 that the polynomials $G_l$ of these operations do not contain any gradable monomials. 
 Therefore such relation cannot exist if $\alpha_i\neq 0$.
\end{proof}

To prove the next Lemma we will use derivatives of poly-operations (Section \ref{prelim}),
and mainly derivatives of Chern classes (Section \ref{chern_constr}). 
Recall that by Cartan's formula $\partial^1 c_n$ is expressible 
as an integral polynomial in Chern classes $c_1,\ldots, c_{n-1}$.

\begin{Lm}\label{polyderint}
Let $\phi\in\QQ [[c_1, c_2, \ldots]]^{\odot N}$ 
be any $N$-ary poly-operation from $\tilde{K}(n)^*$ to $CH^*\ot\Zp$.
Assume that all its first derivatives
can be expressed as $\Zp$-series in Chern classes.

Then $\phi\in \Zp[[c_1, c_2, \ldots]]^{\odot N}$.
\end{Lm}
\begin{proof}
We will prove the Lemma by contradiction.
First, it is enough to treat each graded component $\phi_j:\tilde{K}(n)^*\rarr CH^j\ot\Zp$
 of $\phi$ separately.
Thus, we may assume that $\phi$ is in fact a polynomial and not a series.

Second, if a counter-example to the statement of the Lemma existed, 
then there would be a counter-example $\phi$
 s.t. denominators of $\phi$ are at most $p$. 
Otherwise one can multiply $\phi$ by $p$ to get a counter-example with smaller denominators.
So we assume that $p\phi\in\Zp[c_1, \ldots, c_n]^{\odot N}$ for some $n\ge 1$.

Third, we may assume that all coefficients of $\phi$ are not $p$-integral,
 as we can subtract integral part without breaking a counter-example.

We will now continue to reduce a counter-example in order to finally get
$\sum_{\mathbf{i} = (i_1,i_2,\ldots,i_N); (r_1,\ldots, r_N)} \frac{\beta_{\mathbf{i}}}{p} 
\phi_{i_1}^{p^{r_1}}\odot \phi_{i_2}^{p^{r_2}} \odot \cdots \odot \phi_{i_N}^{p^{r_N}}$,
 where $\beta_{\mathbf{i}}\in\Zp^\times$ and $p^n \nmid i_l$ for all $1\le l \le N$.
This would contradict Cor. \ref{cr_polyadd}.\\

Denote by $i$ the highest index of the Chern class $c_i$ appearing in the first component of $\phi$.
One can write down $\phi$ as a sum 
$\sum_{\mathbf{s} = (s_i, s_{i-1}, \ldots, s_1)} \alpha_{\mathbf{s}} 
c_i^{s_i}c_{i-1}^{s_{i-1}}\cdots c_1^{s_1} \odot P_{\mathbf{s}}$, 
where ${P_{\mathbf{s}}\in \Zp[c_1, \ldots, c_n]^{\odot (N-1)}}$ is an integral polynomial
 in external products of Chern classes,
 s.t. $P_{\mathbf{s}}$ is not zero modulo $p$. 
 
Note that if $\alpha_\mathbf{s}$ is not zero and thus not integral by our assumptions,
then derivatives $\partial_{j-1} P_{\mathbf{s}}$ for any $j, 2\le j\le N$
 are divisible by $p$ as polynomials in Chern classes. 
Indeed, 
the derivative in the $j$-th component of $\phi$ contains
$\alpha_\mathbf{s} c_i^{s_i}c_{i-1}^{s_{i-1}}\cdots c_1^{s_1} \odot \partial_{j-1} P_{\mathbf{s}}$ which is integral
iff $\alpha_\mathbf{s} \partial_{j-1} P_{\mathbf{s}}$ is integral.

\sloppy{Define a (lexico-graphical) order $\succ$ on $i$-tuples of non-negative numbers
 as follows:}\\ $(s_i, s_{i-1}, \ldots, s_1) \succ (r_i, r_{i-1}, \ldots, r_1)$
 if and only if there exist $k:1\le k\le i$
 s.t. $s_j = r_j$ for $i\ge j \ge k$ and $s_{k-1} > r_{k-1}$.
 For example, $(1,0,\ldots,0)\succ (0,2,\ldots,0)$.
 
 Let $\mathbf{r}=(r_i, r_{i-1}, \ldots, r_1)$ be the $\succ$-highest index of non-zero coefficients
  $\alpha_{\mathbf{s}}$
 of $\phi$.
Without loss of generality we may assume that $\alpha_{\mathbf{r}}=\frac{1}{p}$
and by the choice of $i$ we have $r_i\neq 0$. Denote $P:=P_{\mathbf{r}}$.

 We will show that $\phi$ having an integral derivative in the first component
implies strong restrictions on the tuple $\mathbf{r}$.

{\bf Step 1.} Tuple $\mathbf{r}$ has to be of the form $(p^m, 0, 0, \ldots, 0)$ 
for some $m\ge 0$.

Let us recall how to differentiate monomials in Chern classes:
$$(\partial c_i^{r_i}\cdots c_1^{r_1})(x,y)=
(c_i(x)+c_i(y)+\partial c_i(x,y))^{r_i}\cdots (c_1(x)+c_1(y))^{r_1} 
 - \prod_{l=1}^i (c_l(x))^{r_l}  - \prod_{l=1}^i (c_l(y))^{r_l}.$$

In particular, if there exists $r_j\neq 0$, $j<i$,
then the expression on the right contains $c_i(x)^{r_i}\prod_{l=1}^{i-1} c_l(y)^{r_l}$.
This can not cancel, as the derivatives of $c_j$ are polynomials in $c_s$ with $s<j$. 

Thus, if there exists $r_j\neq 0$, $j<i$, then in the derivative  $\partial_1\phi$ 
we get a summand ${\frac{1}{p}c_i^{r_i}\odot \prod_{l=1}^{i-1} c_{l}^{r_{l}}\odot P}$,
which can not be cancelled by derivatives of other summands in $\phi$
due to the highest $\succ$-order of the tuple $\mathbf{r}$.

If $r_i$ is not a power of $p$ then there exist $k:1\le k \le r_i-1$, 
s.t. $\binom{r_i}{k}$ is not divisible by $p$. Thus $\partial^1 (\frac{1}{p} c_i^{r_i}\odot P)$ contains
 non-integral term $\binom{r_i}{k}\frac{1}{p} c_i^{r_i-k}\odot c_i^{k}\odot P$
 which can not be cancelled by derivatives of other monomials in $\phi$,
 because of the highest order of $(r_i, 0,\ldots, 0)$.

We have shown that the highest $\succ$-order term looks like 
this: $\frac{1}{p}c_i^{p^m}\odot P$.

{\bf Step 2.} The number $i$ has to be non-divisible by $p^n$. 

Assume the contrary and denote $v=\frac{i}{p^n}$.
Look at the derivative of $c_i$ as given by the Cartan's formula. 
It contains, in particular, 
the expression $-\frac{a_1}{p}\sum_{j=1}^{p^n-1} \binom{p^n}{j} c_v^j\odot c_v^{p^n-j}$.
 Note that it has a summand which is not zero modulo $p$ (e.g. $j=p^{n-1}$).
Therefore $\partial_1 (\frac{1}{p}c_i^{p^m}\odot P)$ 
contains a non-integral summand 
$\binom{p^n}{p^{n-1}} \frac{a_1}{p^2} c_v^{p^{m+n-1}} \odot c_v^{(p-1)p^{m+n-1}} \odot P$,
which we will call '{\sl bad}'.

We claim that in order for the 'bad' monomial to be cancelled in the derivative,
 $\phi$ itself should contain the monomial
$-\frac{a_1}{p^2}\binom{p^n}{p^{n-1}} \binom{p^{m+n}}{p^{m+n-1}}^{-1} c_v^{p^{m+n}}\odot P$. 
This would contradict our assumption that demoninators are at most $p$ in $\phi$.

The only Chern classes $c_j$, $j\le i$ s.t. $\partial^1 (c_j)^t$ 
contains $b_{(e,g)} c_v^{e}\odot c_v^{g}$ 
for some  $b_{(e,g)}\neq 0, e, g$ and some $t$ are $c_i$ and $c_v$. 
 To see this recall that the formal group law $F_{K(n)}$ 
 has only monomials $x^\alpha y^\beta$ with 
 $\alpha+\beta \equiv 1 \mod (p^n-1)$ (Prop. \ref{prop:mor_grad}, (1)). 
Thus, by Cartan's formula if $\partial^1 (c_j)^k$ contains 
external product $b_{(e,g)} c_v^{e}\odot c_v^{g}$,
then either $j=v$ or $\partial^1 c_j$ contains $c_v^{e'}\odot c_v^{g'}$ for some $e',g'>0$,
s.t. $(e'+g') \equiv 1 \mod (p^n-1)$.
However, $v(e'+g')=j$ has to be less or equal than $i=vp^n$,
and therefore $j=i$.

Therefore a bad monomial in $\partial^1 \phi$ may be cancelled 
only by derivatives of monomials $b_{(e,g)} c_i^ec_v^g\odot P$ for some $e\ge 0$
and $g>0$.
Fix maximal $p^m>e>0$ s.t. $b_{(e,g)}\neq 0$. This coefficient is not integral by our assumptions.
The derivative $\partial_1 (b_{(e,g)} c_i^ec_v^g\odot P)$ contains 
a non-integral monomial $b_{(e,g)}c_i^e\odot c_v^g\odot P$,
which can be cancelled only by the derivative of $b_{(e',g')}c_i^{e'}c_v^{g'}\odot P$
 for some $e'>e$,$g'\ge 0$ and $b_{(e',g')} \in \frac{1}{p}\Zp$. 
 By maximality of $e$ it has to be cancelled by $\frac{1}{p} c_i^{p^m}\odot P$,
 but the derivative of the latter term 
 contains only integral terms of the form $b c_i^e\odot c_v^g\odot P$ with $e>0$, $g>0$.
 
Thus, $e=0$ and so, a 'bad' momonial
 has to be cancelled by the derivative of $b c_v^{p^{n+m}}\odot P$.
As $\partial (c_v^{p^{n+m}})$ contains $c_v^{p^{m+n-1}} \odot c_v^{(p-1)p^{m+n-1}}$
with the coefficient $\binom{p^{m+n}}{p^{m+n-1}}$,
$b$ has to be equal to $-\frac{a_1}{p^2}\binom{p^n}{p^{n-1}} \binom{p^{m+n}}{p^{m+n-1}}^{-1}$
which has a $p$-valuation equal to $-2$. 
As we do not allow such big denominators in our counter-example $\phi$, 
 the claim of Step 2 is proved.

{\bf Step 3.} Reduce the non-integral monomial with the highest lexico-graphical order.

 If $i$ is non-divisible by $p^n$, then in the notaion of Lemmas \ref{lm_add_powers} and \ref{chern_exist} 
$\mu_i=0$. Thus,  a generator $\phi_i$  of additive operations to $CH^i\ot\Zp$  
 is expressible as integral polynomial $c_i-P_i\in\Zp[c_1,\ldots,c_i]$ 
 with a summand $c_i$.
Thus, the derivative of $\psi:= \phi - \frac{\phi_i^{p^m}}{p}\odot P$ in the first component 
is still integral, because $\partial^1 \phi_i=0$, and 
$\partial^1 (\phi_i)^{p^m}=\sum_{k=1}^{p^m-1} \binom{p^m}{k}\phi_i^k\odot \phi_i^{p^m-k}$, which
is divisible by $p$. 
Derivatives of $\psi$ in other components are integral because derivatives of $P$ are $p$-divisible 
as explaine in the beginning of the proof.

Note that the highest $\succ$-order of $\mathbf{r}=(r_i,\ldots, r_1)$ 
of non-zero coefficients $a_{\mathbf{r}}$ in  $\psi$ is smaller than for $\phi$. 

{\bf Step 4.}  Reduce the counter-example to the form 
$\sum_s \frac{\alpha_s}{p} \phi_{i_s}^{p^{r_s}}\odot P_s$.

If $\psi$ is not an integral polynomial 
we apply Steps 1-3 reducing the highest $\succ$-order of non-integral monomials
'in the first component'. 
In the end of this reducing procedure 
 we obtain $\phi - \sum_s \frac{\alpha_s}{p} \phi_{i_s}^{p^{r_s}}\odot P_s$, 
which is an integral polynomial in Chern classes
 and $\sum_s \frac{\alpha_s}{p} \phi_{i_s}^{p^{r_s}}\odot P_s$ has derivatives
  in all components which are integral as polynomials in Chern classes. 
Here $\alpha_s\in\Zp^\times$ and $i_s$ is not divisible by $p^n$.

Note that $\phi$ is an integral poly-operation by the assumpption of the Lemma,
and $\phi - \sum_s \frac{\alpha_s}{p} \phi_{i_s}^{p^{r_s}}\odot P_s$
is an integral polynomial in Chern classes, and therefore an integral poly-operation.
Thus a counter-example to the Lemma is reduced to the case of
 $\sum_{s; r_s} \frac{\alpha_s}{p} \phi_{i_s}^{p^{r_s}}\odot P_s$. 
Indeed, if it were integral, then $\phi$ would be integral as well.

{\bf Step 5.} Apply Steps 1-3 to other components of the poly-operation.

Reducing counter-example of the form $\sum_s \frac{\alpha_s}{p} \phi_{i_s}^{p^{r_s}}\odot P_s$
by the procedure above (applied not to the first component)
does not change expressions in the first component.

Thus, finally we get a counter-example of the form
$\sum_{\mathbf{i} = (i_1,i_2,\ldots,i_N); (r_1,\ldots, r_N)} \frac{\beta_{\mathbf{i}}}{p} 
\phi_{i_1}^{p^{r_1}}\odot \phi_{i_2}^{p^{r_2}} \odot \cdots \odot \phi_{i_N}^{p^{r_N}}$,
 where $\beta_{\mathbf{i}}\in\Zp^\times$ and $p^n \nmid i_l$ for all $1\le l \le N$.

{\bf Step 6.}
If this poly-operation is integral,
then multiplying it by $p$, we get a non-trivial relation in poly-operations modulo $p$:
$\sum_\mathbf{i} \beta_{\mathbf{i}} \phi_{i_1}^{p^{r_1}}\odot \phi_{i_2}^{p^{r_2}} 
\odot \cdots \odot \phi_{i_N}^{p^{r_N}}=0 \mod p$. 
Due to Corollary \ref{cr_polyadd} we have arrived at a contradiction and the Lemma is proved.
\end{proof}

\begin{Prop}\label{prop_chern_base}
Any integral poly-operation $\phi$ from $\tilde{K}(n)^*$ to $CH^*\ot\Zp$
 is uniquely expressible as an integral series in external products of Chern classes.
\end{Prop}
\begin{proof}
It is enough to work with polynomials (not series) of Chern classes, 
i.e. to prove that a poly-operation to a component of $CH^*$ is expressible 
as an integral polynomial in external products of Chern classes. 
Let $\phi$ be an integral $r$-ary poly-operation to $CH^n\ot\Zp$.

By Lemma \ref{chern_polybase} the operation $\phi$ and all its derivatives are 
uniquely expressible as rational polynomials in external products of Chern classes. 
We prove by decreasing induction on $s$ that all $s$-th derivatives are integral.

{\bf Base of induction ($s>>1$).}
Poly-operation $\phi$ can be written as a finite $\QQ$-linear combination
of external products of polynomials in components of the Chern
 character $ch_j$.

Recall that if $\psi$ is an additive operation,
then the $j$-th derivative of $\phi^N$ looks like: 
$\sum \binom{N}{r_1, r_2, \ldots, r_j} \phi^{r_1} \odot \phi^{r_2} \odot \cdots \odot \phi^{r_j}$.
In particular, for $j>N$ we have $\partial^j \phi^N =0$.

As the Chern character is additive, the same applies to polynomials
in it and external products in them. 

One could prove the base step directly from the COT using continuity of poly-operations.

{\bf Induction step ($s \rarr (s-1)$).} 
Suppose that all derivatives $\partial^s \phi$ are integral polynomials in products of Chern classes.
Then $(s-1)$-th derivatives  $\partial^{s-1} \phi$ are
 integral poly-operations for which Lemma \ref{polyderint}
is applicable. Therefore they are integral polynomials in external products of Chern classes.
\end{proof}

This finishes the proof of Theorem \ref{main}.



\begin{thebibliography}{6}

\bibitem{PS_Araki}
S.~Araki \textit{Typical formal groups in complex cobordism and K-theory},
1973.

\bibitem{PS_Kash}
T.~Kashiwabara \textit{Hopf rings and unstable operations},
J. of Pure and Applied Algebra 94, 183-193 (1994).

\bibitem{PS_LevMor}
M.~Levine, F.~Morel \textit{Algebraic cobordism.}
Springer Science \& Business Media (2007).

\bibitem{PS_PanSmi}
I.~Panin, A.~Smirnov \textit{Push-forwards in oriented cohomology theories of algebraic varieties},
K-Theory Preprint Archives (2000),  \# 459.

\bibitem{PS_PetSem}
V.~Petrov, N.~Semenov \textit{Morava K-theory of twisted flag varieties},
arXiv preprint arXiv:1406.3141 (2014).

\bibitem{PS_Vish1}
A.~Vishik \textit{Stable and unstable operations in algebraic cobordism},
arXiv preprint arXiv:1209.5793 (2012).

\bibitem{PS_Vish2}
A.~Vishik \textit{Operations and poly-operations in Algebraic Cobordism},
arXiv preprint arXiv:1409.0741 (2014).

\end{thebibliography}
\end{document}